\documentclass[12pt,twoside,reqno]{amsart}
\usepackage{amsfonts}
\usepackage{amsmath,amscd}
\usepackage{amsthm}
\usepackage{latexsym}
\usepackage{amssymb}
\usepackage{enumerate}
\usepackage{hyperref}
\usepackage{xcolor}
\usepackage{url}
\newtheorem{theorem}{Theorem}
\newtheorem*{theorem*}{Theorem $\mathbf{1'}$}
\newtheorem{proposition}[theorem]{Proposition}
\newtheorem{corollary}[theorem]{Corollary}

\newtheorem{lemma}[theorem]{Lemma}

\newtheorem*{claim*}{Claim}

\newcommand{\varpsilon}{\varepsilon}
\newcommand{\Z}{\mathbb{Z}}
\newcommand{\Q}{\mathbb{Q}}
\newcommand{\R}{\mathbb{R}}
\newcommand{\C}{\mathbb{C}}
\newcommand{\N}{\mathbb{N}}

\newcommand{\e}{\mathrm{e}}
\newcommand{\re}{\mathrm{Re}}
\newcommand{\im}{\mathrm{Im}}
\newcommand{\ie}{{\it{i.$\,$e.\ }}}
\newcommand{\InZ}{I_{Z+1}^n} 
\newcommand{\Hh}{\mathcal{H}}
\newcommand{\M}{\mathcal{M}}

\newcommand{\g}{\mathfrak{g}}

\begin{document}
\title{Polynomials  associated to Lie algebras}
 \author{Mat\'ias Bruna, Alex Capu\~nay and Eduardo Friedman}
 \address{Department of Mathematics, Univ.\  Toronto, Toronto, ON, M5S 2E4, Canada}\email{matias.bruna@mail.utoronto.ca}

 \address{Fac.  Ciencias e ICEN,  Univ.\  Arturo Prat, Av.\ Arturo Prat 2120, Iquique, Chile}\email{acapunay@unap.cl}

 \address{Depto.\   Matem\'aticas, Fac.\ Ciencias, Univ.\  Chile, Casilla 653, Santiago, Chile}  \email{friedman@uchile.cl}

\begin{abstract} 
We associate to a semisimple complex Lie algebra $\mathfrak{g}$ a sequence of polynomials 
$P_{\ell,\mathfrak{g}}(x)\in\mathbb{Q}[x]$ in $r$ variables, where $r$ is the rank of $\mathfrak{g}$ and $\ell=0,1,2,\ldots $.  The  polynomials  $P_{\ell,\mathfrak{g}}(x)$
  are uniquely associated to the isomorphism class of $\mathfrak{g}$, up to re-numbering the variables, and are defined as special values of a variant of Witten's zeta function. 
Another set of polynomials associated to $\mathfrak{g}$ were defined in 2008 by Komori, Matsumoto and Tsumura using different special values of another variant of Witten's zeta function.
\end{abstract}
 
 \maketitle


\section{Introduction}\label{INTRO}
 Motivated by  physics,  Witten introduced in 1991  the Dirichlet series 
$
\zeta_\text{W}(s;G):= \sum_{\rho }\frac{1}{(\dim\rho)^{s}}$
 \cite[eq.\ 4.72, p.\ 197]{Wit}, where the sum runs over all 
 irreducible unitary representations $\rho$ of certain groups  $G$. Witten used the values of $\zeta_\text{W}(s;G)$  at  positive integers $s$ to give formulas for volumes of  some moduli spaces of principal $G$-bundles.  

When $G$ is a simply connected compact Lie group,   the correspondence between representations of $G$ and of its Lie algebra $\g$    
 led  Zagier \cite{Zag} to   the  expression 
\begin{align} \label{WittDef}
\zeta_\text{W}(s;G) =K_\g^s \sum_{m\in\N^r}\prod_{\  \alpha\in\Phi^+}\!(m_1\lambda_1+\cdots +m_r\lambda_r,\alpha^{\!\lor})^{-s}=:\zeta_\text{W}(s;\g) ,
\end{align}
where  $r$ is the rank of $\g$, $\re(s)>r$, $\alpha$ runs over a set $\Phi^+$ of positive roots in a root system $\Phi$ associated to $\g$, $(\ \, ,\  )$ denotes the inner product (Killing form),  $\alpha^{\!\lor}: =\frac{2}{(\alpha,\alpha)}\alpha$ 
is the co-root corresponding to $\alpha$,    $\lambda_1,\ldots,\lambda_r $ are the fundamental dominant weights 
 associated  to $\Phi^+ $, and $K_\g  :=\prod_{\alpha\in\Phi^+}(\lambda_1+\cdots + \lambda_r,\alpha^{\!\lor}) \in\N$.  Zagier also remarked  that 
 in the case of  $\g=\mathfrak{sl}_2$, the function  $\zeta_\text{W}(s;\g) $ coincides with   the Riemann zeta function $\zeta(s)$. 

 No polynomials are in sight when considering just $\zeta_\text{W}(s;\g)$,   but recall that       Hurwitz  inserted a    variable $x$  into $\zeta(s)$ by defining  
 $$
\qquad H(s,x):=\sum_{k\in \N_0} 
    (x+k)^{-s}\qquad\qquad(x>0,\ \re(s)>1,\ 
\N_0:=\N\cup\{0\}). $$
Thus,  $H(s,1)=\zeta(s)$.     As with $\zeta(s)$, there is an analytic   continuation  of $H(s,x)$
 to all $s\in\C- \{1\}$ whose values $H(-\ell,x)$ at $s=-\ell$ for  $\ell\in\N_0$   are polynomial  functions of  $x$. In fact, 
$H(-\ell,x)=-B_{\ell+1}(x)/(\ell+1)$ is  the  Bernoulli polynomial of degree $\ell+1$, with a different  normalization. 

Here we    extend the Hurwitz  procedure to  semisimple Lie algebras and define polynomials $P_{\ell,\g}(x)$ in 
 $r$ variables,  where   $r$ is  the rank of $\g$  and 
 $\ell\in\N_0.$  These polynomials are  naturally associated to $\g$ since they turn out to depend only on the isomorphism class of $\g$, up to re-numbering the variables $x_1,\ldots,x_r$. 
To define $P_{\ell,\g}$     start with  $\re(s)>r$ and  
 $x=(x_1,\ldots,x_r) \in(0,\infty)^r$,   and      define the 
  absolutely convergent Dirichlet series   (again with $\N_0:=\N\cup\{0\})$ 
\begin{align} 
\label{GenHurwDef}
\zeta_\g(s,x) :=  \sum_{m\in\N_0^r}\prod_{\ \alpha\in\Phi^+}\big((m_1+x_1)\lambda_1+\cdots+ (m_r+x_r)\lambda_r, \alpha^{\!\lor}\big)^{-s}  .
\end{align}
Thus, $K_\g^s\zeta_\g\big(s,(1,\ldots,1)\big) 
 =\zeta_\text{W}(s;\g)$.
It is known (see Prop.\ \ref{MeromZeta})  that   $s\to\zeta_\g(s,x)$ has a meromorphic continuation to all $s\in\C$  which is  
 regular at $s=0,-1,-2,\ldots$. 

Our main aim here is to prove the following.
 \begin{theorem}\label{Liepolys} Let $\g$ be a 
 semisimple complex Lie algebra of rank $r$,  
let $n $  be the number of positive roots 
 in a root system  for $\g$, let $\ell=0,1,2,\ldots$, and let 
 $\zeta_\g(s,x) $ be as in \eqref{GenHurwDef}. Then  $P_{\ell,\g}(x):=\zeta_\g(-\ell,x)$  is a   polynomial 
with rational coefficients, has total  degree $n\ell+r$   in $x=(x_1,\ldots,x_r)$, and satisfies the following properties. 
\vskip.2cm

\noindent$\mathrm{(0)}$    $P_{\ell,\mathfrak{sl}_2 }(x)=-B_{\ell+1}(x)/(\ell+1)$, where   $B_{\ell+1}(x)$ is the   ${(\ell+1)}^{\mathrm{th}}$-Bernoulli polynomial. 
    \vskip.15cm 

\noindent$\mathrm{(i)}$   $P_{\ell,\g}(x)$ depends only on the isomorphism class of $\g$, up to re-numbering    $x_1,\ldots,x_r$. 

    \vskip.15cm 

\noindent$\mathrm{(ii)}$  If $\g_1$ and $\g_2$ are semisimple Lie  algebras, then $P_{\ell,\g_1\times\g_2}(x,y)=P_{\ell,\g_1}(x)  P_{\ell,\g_2}(y)$, on conveniently numbering the variables.

    \vskip.15cm

\noindent$\mathrm{(iii)}$  Define commuting  difference operators  
 $(\Delta_{e_k} P)(x):=P(x+e_k)-P(x)   $, where $e_1,\ldots,e_r$  is the standard basis of $\R^r$. Then 
$$
\big( \Delta_{e_1}\circ\Delta_{e_2}\circ\cdots\circ\Delta_{e_r})(P_{\ell,\g}\big)(x)=(-1)^r\Big(\!\!\prod_{\ \alpha\in\Phi^+} \sum_{k=1}^r  x_k( \lambda_k,\alpha^{\!\lor})\Big)^{\!\ell}\in\Z[x].
$$

    \vskip.15cm 
\noindent$\mathrm{(iv)}$  $P_{\ell,\g}(\mathbf{1} - 
x)=(-1)^{n\ell+r}P_{\ell,\g}(x),$ where   $\mathbf{1}:=(1,\ldots,1)\in \R^r$.  
    \vskip.15cm 

\noindent$\mathrm{(v)}$ There is  a   Bernoulli  polynomial  expansion
$$  
\quad\quad P_{\ell,\g}(x)=\sum_{\substack{L=(L_1,\ldots,L_r)\in\N_0^r\\ L_1+\cdots+ L_r=n\ell+r}}a_L\prod_{i=1}^r B_{L_i}(x_i)\qquad\qquad(a_L=a_{L,\ell,\g}\in\Q,\ \N_0:=\N\cup\{0\}).
$$
\end{theorem}
\noindent The caveat   in (i) and (ii) of 
 Theorem \ref{Liepolys} about re-numbering the variables is due   to the arbitrary choice of numbering  of  the fundamental dominant 
 weights $\lambda_1,\ldots,\lambda_r $. 

Recall that   Bernoulli polynomials satisfy the identities 
$$B_{\ell+1}(x+1)-B_{\ell+1}(x)=(\ell+1)x^\ell,\qquad\quad B_{\ell+1}(1-x) = 
(-1)^{\ell+1} B_{\ell+1}(x) .
$$
In view of property (0), (iii-v)  in 
 Theorem \ref{Liepolys}   generalize the above identities   from $\mathfrak{sl}_2$    to any semisimple $\g$. It  is also clear that (v) implies (iv).

In contrast with the case of rank $r=1$,  when $r>1$ properties (iii) and (v)   no longer  uniquely characterize the polynomial $P_{\ell,\g}$. They only fix the $a_L$ for $L$ such that  $L_i\not=0$ for all $i$.  It would be interesting to find a clear characterization of 
$P_{\ell,\g}$  in terms of the root system attached to $\g$.   A 
 property   of the $P_{\ell,\g}$ polynomials  additional to  Theorem \ref{Liepolys}  is provided by K.$\,$C.\ Au's recent 
 proof \cite{Au} of the Kurokawa-Ochiai  conjecture \cite{KO}, \ie  $P_{  \ell,\g}(\mathbf{1})=0 $ for all  even $\ell\in\N$.

 Only for $\g=\mathfrak{sl}_3$ have we   been able to prove  a relatively simple formula for $P_{\ell,\g}$ for all $\ell\in\N_0$. 
 Although  we shall not prove this here,
 \begin{align}\nonumber
 P_{\ell,\mathfrak{sl}_3}(x_1,x_2) =\frac{ (\ell !)^2\big( B_{3\ell+2}(x_1)+ B_{3\ell+2}(x_2) \big)}{ 2(-1)^{\ell}(3\ell+2)(2\ell+1)!}  
+\sum_{k=0}^{\ell}    \binom{\ell}{k}  \frac{B_{2\ell-k+1}(x_1)B_{\ell+k+1}(x_2)}{(2\ell-k+1)(\ell+k+1)},
\end{align}
where   $ \binom{\ell}{k}$ denotes a  binomial coefficient.  
In Theorem \ref{AppliedFiniteTaylor} we actually give a formula for  $P_{\ell,\g}$, but it is too complicated to be   more than an algorithm  
  for computing $P_{\ell,\g}$, and practical  only for small $r$ and  $\ell$.

The definition and study of polynomials associated to semisimple 
 Lie algebras via variants of Witten's zeta function was initiated nearly 20 years ago by Komori, 
 Matsumoto and  Tsumura.\footnote{\ See  \cite{KMT1}  for an early summary of their work and their    recent book  \cite{KMT2} on  zeta functions associated to root systems for a comprehensive treatment.}   Because they were mainly interested in the values   at positive 
 integers, and also at $n$-tuples of positive integers,  they inserted a  vector  variable  $y\in \R\lambda_1+\cdots+\R\lambda_r$ into 
 \eqref{WittDef} differently  than we did in \eqref{GenHurwDef}. Namely they  defined for 
 $\mathbf{s}=(s_\alpha)_{\alpha\in\Phi^+}\in\C^n$ with  $\re(s_\alpha)$ sufficiently large, 
\begin{align} \label{KMTDef}
S( {\mathbf{s}},y;\g):=\sum_{m\in\N^r}\e^{2\pi i(y,\sum_{k=1}^r m_k\lambda_k)}\prod_{\alpha\in\Phi^+}\big(\sum_{k=1}^r m_k\lambda_k,\alpha^{\!\lor}\big)^{-s_\alpha}   . 
\end{align} 

 The function $y\to S( {\mathbf{s}},y;\g)$  is   not quite a polynomial in $y$ (for any fixed $\mathbf{s}$) since it has the   periodicity $S(\mathbf{s},y+ \alpha^{\!\lor};\g)=S(\mathbf{s},y;\g)$ for all  $\alpha\in\Phi$. However,  Komori, Matsumoto and  Tsumura \cite{KMT1}  \cite{KMT2} showed  that  if we take $s_\alpha\in\N  $  and exclude $y$ from a set of measure 0, then $S( {\mathbf{s}},y;\g)$ is locally a polynomial in $y$. 
The simplest     of these 
KMT polynomials occur  for 
 $\g=\mathfrak{sl}_2$, where they are essentially the Bernoulli polynomials. It might be interesting to study   how the $P_{\ell,\g}$ are related to the KMT polynomials for other $\g$ (cf.\ \cite[\S17.2]{KMT2}).

The polynomials    $P_{\ell,\g}$  are closely related  to another   set of polynomials  arising from 
\begin{align} \label{GenIntzetaDef}
\mathcal{Z}_\g(s,x)&:= \int_{t\in(0,\infty)^r} \prod_{\ \alpha\in\Phi^+}\big((t_1+x_1)\lambda_1+\cdots+ (t_r+x_r)\lambda_r, \alpha^{\!\lor}\big)^{-s} \,dt   ,
\end{align}
where again we initially assume $\re(s)>r$ and $x\in(0,\infty)^r$. 
Like  $\zeta_\g(s,x)$ in \eqref{GenHurwDef}, $\mathcal{Z}_\g(s,x)$ has a meromorphic continuation in $s$ to all of $\C$ which is regular  at $s=-\ell$ for $\ell\in\N_0$ (see   
Proposition \ref{MeromZ}). This  allows us to 
 define $Q_{\ell,\g}(x):=\mathcal{Z}_\g(-\ell,x)$, which  turns out to be   a homogeneous polynomial in $x$ of total 
 degree $n\ell+r$. 

On ordering the variables compatibly, the $Q_{\ell,\g}$  and $P_{\ell,\g}$ are related  by the  Raabe formula  (cf.\ \cite[Prop. 2.2]{FP}) 
\begin{align} \label{RaabeFORM}
Q_{\ell,\g}(x) = 
\int_{t\in[0,1]^r} P_{\ell,\g}(x+t)  \,dt.
\end{align}
In fact,   \eqref{RaabeFORM} is equivalent  to  \cite[Lemma 2.4]{FP}
\begin{align}\label{RaabeFORM2}
Q_{\ell,\g}(x) =\sum_{\substack{L=(L_1,\ldots,L_r)\in\N_0^r\\ L_1+\cdots+ L_r=n\ell+r}}a_L\prod_{i=1}^r x_i^{L_i},
\end{align}
where   $a_L=a_{L,\ell,\g}$ is given by  (v) of Theorem \ref{Liepolys}. 
The map in \eqref{RaabeFORM} taking $P$ to $Q$, namely $Q (x) =  \int_{t\in[0,1]^r} P (x+t)  \,dt$, is  an 
 automorphism of  $\R[x]$ only as a graded $\R$-vector space. It certainly is not a ring automorphism of $\R[x]$. Thus, $Q_{\ell,\g}$ and 
$P_{\ell,\g}$  should have very different properties, even if they are both naturally associated to $\g$ and are  easily computed from one another.

Except for (i) and (ii) in Theorem \ref{Liepolys}, the 
 remaining properties stated  there are shared by more general series and integrals.  We devote \S\ref{ShinBarnes}--\ref{Relations} to studying these 
 functions under assumptions that allow us to treat $\zeta_\g$ in Theorem  \ref{Liepolys}. In \S\ref{ProofTheorem1} we prove     Theorem $1'$, which  
 includes Theorem \ref{Liepolys} and  results on the $Q_{\ell,\g}$ polynomials. 
In the final section we use  Theorem \ref{AppliedFiniteTaylor} to compute examples of $P_{\ell,\g}$ for $\g$ of small rank. We also take $\ell$  small  to avoid long expressions.

\section{The  Shintani-Barnes zeta function 
 $\zeta_{N,n}$}\label{ShinBarnes}
Let $\M=(a_{ij})_{\substack{1\le i\le N\\ 1\le j\le n}}$ be an  $N\times n$ matrix 
with coefficients $a_{ij}\in\C$.   We  henceforth always assume that $\M$ satisfies  
\begin{align}
 &\textbf{Hypothesis $\Hh$.}\text{\ Each entry $a_{ij} $  of $\M$  
     either vanishes or has a positive real part,}\nonumber\\
&    \text{and   no row vanishes.} \label{HypH}
\end{align}

\noindent Thus, for each $i$ there is a $j$ such that  $\re(a_{ij})>0$.   
We let $Z_\M$ be such that  every row of $\M$ has at least $n-Z_\M $ non-zero  entries, and some row has exactly $n-Z_\M$ such entries. Letting $z(i):=\text{cardinality}\big(\{j|\,a_{ij}=0\}\big)$, we have by Hypothesis $\Hh$
\begin{equation}\label{zM}
  0\le Z_\M:=\max_i\{ z(i) \}<n .
\end{equation}

  For $w=(w_1,\ldots,w_n)\in\C^n$   such that $\re(w_j)>0\ \,(1\le j\le n)$   define  for 
$\re(s)>N/(n-Z_\M) $ the absolutely  convergent  series   and   integral    (see \S\ref{Converge})  
\begin{align} \label{ZetaSerDef}
\zeta_{N,n}(s,w,\M)&:= \sum_{k_1,\ldots,k_{N}=0}^{\infty}\, \prod_{j=1}^{n}\! \big((w_{j}+k_1a_{1j} +k_2a_{2j}+\cdots+k_{N}a_{Nj})^{-s}\big),  
\\  
\mathcal{Z}_{N,n}(s,w,\M)&:=\int_{t\in(0,\infty) ^N}  \prod_{j=1}^{n}\! 
\big((w_{j}+t_1a_{1j}+t_2a_{2j}+\cdots+t_{N}a_{Nj})^{-s}\big)\,dt , \label{ZetaIntDef}
\end{align}
where  the powers in each factor  use the principal branch of the logarithm  and $dt=dt_1\cdots dt_N$ is Lebesgue measure. 

The function  $\zeta_\g(s,x)$ defined in 
 \eqref{GenHurwDef} is  a special case of $\zeta_{N,n}(s,w,\M )$   in \eqref{ZetaSerDef} as 
\begin{align} \label{LieToShinBarn}
&\zeta_\g(s,x)= \zeta_{r,n}(s,W(x),\M_\g),\qquad r:=\text{rank}(\g) ,\qquad n :=\text{cardinality}(\Phi^+)  ,
\\ &   \big(W  (x)\big)_\alpha :=\sum_{i=1}^r x_i(\lambda_i,\alpha^{\!\lor}) ,\qquad \big(\M_\g\big)_{i\alpha}:=(\lambda_i,\alpha^{\!\lor})\ \ \ (1\le i\le r,\ \alpha\in \Phi^+) ,   
 \qquad  \nonumber
\end{align} 
where $x\in(0,\infty)^r$, and we have  labeled the  $n$ columns of $\M_\g$ by $\alpha\in\Phi^+$  instead of $j$ (the order of the factors in   \eqref{ZetaSerDef} changes nothing, of course). 
  Hypothesis $\Hh$ is satisfied since 
 $\big(\M_\g\big)_{i\alpha}\in \N\cup\{0\} $  and 
$\M_{i\alpha_i}=1$, 
 where $\alpha_i\in\Phi^+$ is the simple root satisfying $(\lambda_i, \alpha_j^{\, \lor}   )=\delta_{ij}$, the Kronecker delta \cite[p.\ 67]{Hum}. 
Moreover, $ \big(W  (x)\big)_\alpha>0$ since   $\alpha=\sum_{i=1}^ r d_i\alpha_i$ with $d_i\ge0$ and some $d_{i_0}>0$.
Similarly,   from \eqref{ZetaIntDef} 
and  \eqref{GenIntzetaDef}, 
\begin{align} \label{LieToInt}
\mathcal{Z}_\g(s,x)=\mathcal{Z}_{r,n}(s,W(x),\M_\g) .
\end{align}

\subsection{Half-plane of convergence}\label{Converge}

 The    absolute convergence of the series in \eqref{ZetaSerDef} and of the  integral in \eqref{ZetaIntDef}, uniform for $(s,w)$ in compact 
 subsets of   $$\{s|\,\re(s)>N/(n-Z_\M)\}\times \{w\in\C^n|\,\re(w_k)>0,\ 1\le k\le n\},$$ follows  readily from Hypothesis $\Hh$ in \eqref{HypH}. Indeed, let 
$$
c:=\min_{i,j}\{\re(a_{ij})|\, a_{i,j}\not=0\}, \    d:=\min_{j}\{\re(w_j)\}, \  C:=\min(c,d),\   
A_j:=\big\{i\big|\,a_{i,j}\not=0\big\}.
$$ 
Note that $C>0$ by   $\Hh$. Thus, for $\ell_i \ge0\ \,(1\le i\le N)$, 
$$
 \Big|w_j+\sum_{i=1}^N \ell_i a_{ij}\Big|\ge\re\Big(w_j+\sum_{i=1}^N \ell_i a_{ij}\Big)\ge d+c\sum_{i\in A_j}\ell_i\ge C\Big(1+\sum_{i\in A_j}\ell_i\Big),
$$
and so 
\begin{equation}\label{Compare}
\prod_{j=1}^{n}|w_{j}+{\textstyle{\sum_{i=1}^N \ell_ia_{ij}}} |\ge C^n\prod_{j=1}^n \big(1+\sum_{i\in A_j}\ell_i\big)\ge C^n(1+\ell_1^{n-Z_\M}+\cdots +\ell_N^{n-Z_\M}),
\end{equation}
as every $i$ belongs to at  least $n-Z_\M$ different  $A_j$'s by definition \eqref{zM}.  Since 
$$
|z^s|=|z|^{\re(s)} e^{-\im(s)\arg(z)}\ge|z|^{\re(s)} e^{-D\pi/2}\qquad(\re(z)>0, \ |\im(s)|\le D),
$$
it follows from \eqref{Compare}  that 
the series  \eqref{ZetaSerDef} (resp.,  integral  \eqref{ZetaIntDef}) can be compared with a well-known   series (resp., integral)   converging for $\re(s)>N/(n-Z_\M)$. In particular, $\zeta_{N,n}(s,w,\M)$ and 
$\mathcal{Z}_{N,n}(s,w,\M)$ converge if $\re(s)>N,\ \re(w_k)>0  \  (1\le k\le n)$, and are analytic functions of $(s,w)$ in this domain.

\subsection{Analytic continuation of the zeta integral $\mathcal{Z}_{N,n}$}\label{Zmercont} We now turn to 
 the meromorphic continuation of the zeta  integral 
 $\mathcal{Z}_{N,n}$ in  \eqref{ZetaIntDef}, leaving the 
 Dirichlet series $\zeta_{N,n}$ in 
 \eqref{ZetaSerDef} to \S\ref{Smercont}. We will generalize the approach of  \cite[\S2]{FR}. 

Pick and fix  an integer 
 $Z$   satisfying $Z_\M\le Z <n$, where $Z_\M$ was defined in \eqref{zM}. We will be 
 interested in $Z=Z_\M$, but no complications   arise from allowing larger 
 values of $Z$. As $N/(n-Z)\ge N/(n-Z_M)$,  \S\ref{Converge} implies that    $\zeta_{N,n}(s,w,\M)$ and  $\mathcal{Z}_{N,n}(s,w,\M)$ converge for $\re(s)>N/(n-Z)$.

Applying 
$a^{-s}\Gamma(s)=\int_0^\infty t^{s-1} e^{-at}\,dt\ \,(\re(a)>0,\ \re(s)>0)$ 
to  \eqref{ZetaIntDef} we find 
\begin{align}\nonumber 
\Gamma(s)^n \mathcal{Z}_{N,n}&(s,w,\M) = \int_{t\in (0,\infty)^N}\int_{T\in(0,\infty)^n} \prod_{j=1}^n  T_j^{s -1} e^{-T_j( w_j+t_1a_{1j}+\cdots+t_N a_{Nj})}\,dT\,dt\\ 
&=\int_{T\in(0,\infty)^n}\! \Big(\prod_{j=1}^n e^{-w_j T_j}\,T_j^{s-1}\Big) \int_{t\in(0,\infty)^N}\prod_{i=1}^{N} e^{-t_i\sum_{j=1}^n a_{ij}T_j}\,dt\,dT \nonumber\\
&=\int_{T\in(0,\infty)^n}\frac{\prod_{j=1}^n e^{-w_j T_j}\,T_j^{s-1}}{\prod_{i=1}^{N}\big(\sum_{j=1}^n a_{ij}T_j\big)}\,dT =:  \!\!\int_{T\in(0,\infty)^n}\!\! H(T,s,w,\M)\,dT,\label{ZwithGamma}
\end{align}
where $\re(s) >N/(n-Z )$ is assumed and $H$ stands for the integrand to its left. 

For a positive integer  $q \le n$, let $I_q^n$ be the set of injective functions 
 from $\{1,\ldots,q\}$ to $\{1,\ldots,n\}$. 
We regard $I_q^n \subset S_n=I_n^n$ by  requiring that $\gamma(q+1),\ldots,\gamma(n)$ 
be the elements of $\{1,\ldots,n\}\setminus\{\gamma(1),\ldots,\gamma(q)\}$ 
 listed in increasing order.\footnote{\ This is only for definiteness. Any   ordering of these $n-q$ numbers would do just as well below.} For $\gamma\in I_q^n$  let 
\[\Delta^\gamma\!:=\!\big\{ (T_1,\ldots,T_n)\in(0,\infty)^n  \mid T_{\gamma(1)}>\cdots> T_{\gamma(q)} ,\text{ and } T_{\gamma(q)}> T_{\gamma(l)} ~\text{ for } q<l\le n\big\}.\]
Up to sets of measure 0, $(0,\infty)^n =\bigcup _{\gamma\in  I_q^n}\Delta^\gamma$, and the union is disjoint. 

Picking $q:=Z+1$ and using \eqref{ZwithGamma}   we can write  
\begin{align}\nonumber
&\Gamma(s)^{n}\mathcal{Z}_{N,n}(s,w,\M)=\! \! \sum_{\gamma\in\InZ} \! \! \!  \int_{T\in\Delta^{\gamma}} \! \! \! H(T,s,w,\M)\,dT=\! \! \sum_{\gamma\in\InZ} \! \! \! \int_{T\in\Delta}\! \!H(T,s,w^{\gamma},\M^{\gamma})\,dT,\\ \label{InjDecomp}
&\Delta := \big\{(T_1,\ldots,T_n)\in (0,\infty)^n \mid\,T_1>\cdots> T_{Z+1},\ \ T_{Z+1}> T_\ell~\text{ for }\ell \ge Z+2\big\},\\ \nonumber 
&w^\gamma  :=(w_{\gamma(1)},\ldots,w_{\gamma(n)}),\quad \ \M^\gamma:=(a_{i\gamma(j)}),\quad \ \re(s)>\frac{N}{ n-Z},\quad \   Z_\M\le Z<n .  \end{align}
 As $\M$ satisfies Hypothesis $\Hh$  in \eqref{HypH} if and only if $\M^\gamma$ does and 
 $Z_\M=Z_{M^\gamma}$, \eqref{InjDecomp} shows that it suffices to 
analytically continue  $\int_{T\in \Delta } H(T,s,w ,\M)\,dT$ for all $w$ satisfying $\re(w_j)>0$ $ (1\le j\le n)$ and for all $\M$ satisfying $\Hh$.  

For each $j\ \,(1\le j\le n)$ let $F_j$ be the set of indices $i$ of rows   of $\M$ starting with exactly  $j-1$  zeroes.
Thus,   
\begin{equation}\label{DefFjM}
F_j=F_j(\M):= \big\{i\in\{1,2,\ldots,N\}\mid  a_{ik}=0\ \text{for}\ 1\le  k<j, \  a_{ij}\not=0\big\}.
\end{equation}
Since we have assumed that no row has more than $Z$ zeros, 
\begin{equation}\label{UnionFjM}
\{1,2,\ldots,N\}=\bigcup_{j=1}^n F_j,  \qquad F_j\cap F_{j'}=\varnothing\ \text{ for } j\not=j',\qquad F_j=\varnothing\ \text{ for } j>Z+1.
\end{equation}
We now change variables in \eqref{InjDecomp} from $ T\in\Delta $  to $ \sigma\in(0,\infty)\times(0,1)^{n-1}$ by letting
\begin{align}\label{NewVariables}
\sigma&= (\sigma_1, \sigma_2,\ldots,\sigma_n )=:(\sigma_1,\sigma'),\qquad\sigma_k:=
\begin{cases}
T_1&\  \text{if }  k=1,\\
\frac{T_k}{T_{k-1}}&\  \text{if }2\le k\le Z+1,\\
\frac{T_{k}}{T_{Z+1}}&\  \text{if }Z+2\le k\le n .
\end{cases} 
\end{align}
We can write $T $ in terms of $\sigma$  as
\begin{align}\label{Tfromsigma}
T_k= 
\begin{cases}
T_1\cdot\frac{T_2}{T_1}\cdot\frac{T_3}{T_2}\cdots\frac{T_k}{T_{k-1}}=  \prod_{j=1}^k\sigma_j&\  \text{if }1\le k\le Z+1,\\
 \sigma_k\cdot T_{Z+1}=\sigma_k\cdot\prod_{j=1}^{Z+1}\sigma_j&\  \text{if } Z+2\le k\le n .
\end{cases}  
\end{align}
Hence $\frac{\partial T_k}{\partial\sigma_j}=0$ for $j>k$, which implies that  the Jacobian determinant $J$  is simply
\[
J= \prod_{k=1}^n\frac{\partial T_k}{\partial\sigma_k}  =
\Big(\prod_{k=1}^{Z+1}
  \prod_{j=1}^{k-1}\sigma_{j}\Big)\cdot\Big(\prod_{k=Z+2}^n  \prod_{j=1}^{Z+1}\sigma_{j}\Big)=\prod_{j=1}^{Z+1}\sigma_{j}^{n-j},
\]
where the last equality  uses induction on $n\ge Z+1$.
As $T_k,\sigma_j>0$,   \eqref{Tfromsigma} yields 
\[
\prod_{k=1}^{n}T_{k}^{s-1}=\Big(\prod_{k=1}^{Z+1}
\prod_{j=1}^{k}\sigma_{j}^{s-1}\Big) 
\Big(\prod_{k=Z+2}^{n}\! \! \!\sigma_k^{s-1} \prod_{j=1}^{Z+1}\sigma_{j}^{s-1}\Big)=\Big(\prod_{j=1}^{Z+1}\sigma_{j}^{(1+n-j)(s-1)}\Big)\Big(\prod_{j=Z+2}^{n}\! \!\sigma_j^{s-1}\Big).\]
Using \eqref{UnionFjM} and writing $|F_j|$ for the cardinality of $F_j(\M)$ in \eqref{DefFjM}, we get
\begin{align*}
\prod_{i=1}^{N}{\textstyle{\big(\sum_{j=1}^{n}{a_{ij}T_{j}}\big)}}&=\prod_{j=1}^{Z+1}\prod_{i\in F_{j}}\!{\textstyle{\big(\sum_{k=1}^{n}{a_{ik}T_{k}}}}\big)=\prod_{j=1}^{Z+1}\prod_{i\in F_j}\!\big(a_{ij}T_j+{\textstyle{\sum_{k=j+1}^n a_{ik}T_k\big)}}\\
&=\prod_{j=1}^{Z+1}T_j^{|F_j|} \prod_{i\in F_j}\! \big(a_{ij}+{\textstyle{\sum_{k=j+1}^{n}a_{ik}\frac{T_k}{T_j} \big)}}\\
&= y(\sigma')\cdot\prod_{j=1}^{Z+1}\sigma_{j}^{\sum_{k=j}^{Z+1}|F_{k}|},
\end{align*}
where $\sigma':=(\sigma_2,\ldots ,\sigma_n)$ and $y(\sigma') =y_{\M,Z}(\sigma') $ is given by
\begin{align} \label{yDef}
 y(\sigma'):=\prod_{j=1}^{Z+1}\prod_{i\in F_{j}}\!\big(a_{ij}+{\textstyle{\sum_{k=j+1}^{Z+1}a_{ik}\prod_{r=j+1}^{k}
\sigma_r+(\sum_{k=Z+2}^{n}a_{ik}\sigma_k)	\prod_{r=j+1}^{Z+1}\sigma_{r}\big)}}.
\end{align}
With $H$  as in \eqref{ZwithGamma}, let 
\begin{align} \label{IDef}
I(s)=I(s,w)=I_{\M,Z}(s,w):=\int_{T\in \Delta } H(T,s,w,\M)  \, dT.
\end{align}
From  our change of variable  computations, 
 and   $\sum_{j=1}^{Z+1}  |F_{j}|=N$ $\big($see  \eqref{UnionFjM}$\big)$,  we 
 obtain  
\begin{align}\label{BasicInt}
&I(s)=I(s,w)=\int_{\sigma_{1}=0}^{\infty}\!\!\sigma_{1}^{ns-N-1}  
 e^{-\sigma_{1}w_{1}}\!\int_{\sigma'\in(0,1)^{n-1}}\!\!g(\sigma ) \prod_{j=2}^n\sigma_{j}^{s_j-1} \  d\sigma_n\cdots d\sigma_2   d\sigma_1, 
\\
&  s_{j}:=
\begin{cases}
(n+1-j)s-\sum_{k=j}^{Z+1}|F_{k}| &\  \text{if }1\le j\le Z+1,\\
s & \  \text{if }Z+2\le j\le n.
\end{cases} \label{sjDef}
\\
&g(\sigma )=g_{w',\M,Z}(\sigma_{1},\sigma'):=\frac{\prod_{j=2}^{Z+1} 
e^{-w_{j}\sigma_{1}\sigma_{2}\cdots\sigma_{j}}\cdot
\prod_{\ell=Z+2}^{n}
e^{-w_\ell\sigma_\ell\sigma_1\sigma_{2}\cdots\sigma_{Z+1}}}{y(\sigma')},\label{gDef}
\end{align}
where $  w= (w_1,w_2,\ldots,w_n)=:(w_1,w')$, so  $g$   depends neither  on $w_1$ nor on  $s$. Note that  $s_1=ns-N$, independently of the pattern of zero entries of $\M$ $\big($see
 \eqref{UnionFjM}$\big)$.

\begin{lemma}\label{MeromorphyZ} Assume $\M$ satisfies Hypothesis 
 $\Hh$  in \eqref{HypH},  $Z<n$ is a non-negative  integer such that 
no row of $\M$ has more than $Z $ vanishing entries, and    let $s_j$ be as in \eqref{sjDef}. 
Then $I(s,w)$ in \eqref{BasicInt} is analytic   for $\re(s)>\frac{N}{n-Z}$ and $\re(w_k)>0$, has a meromorphic continuation to $(s,w)\in\C \times\{w\in\C^n|\,\re(w_k)>0,\ 1\le k\le n\}$,  and
\begin{align}\label{IMeromorphy}
\frac{I(s,w)}{\Gamma(ns-N) } \prod_{p=0}^{M}\prod_{j=2}^n (p+s_j) 
\end{align}
is analytic in $(s,w)$ for    any integer $M\ge N$  provided  $\re(s)>\frac{(N-M)}{n}$ and $\re(w_k)>0$.
\end{lemma}
\noindent Assuming the lemma for now, we deduce the meromorphic continuation of $\mathcal{Z}_{N,n} $.
\begin{proposition}\label{MeromZ} If $\M$ and  $Z$ are as in Lemma $\mathrm{\ref{MeromorphyZ}}$, then  
 $(s,w)\to\mathcal{Z}_{N,n}(s,w,\M)$  in \eqref{ZetaIntDef} has a meromorphic continuation  to 
$\C \times\{w\in\C^n|\,\re(w_k)>0,\ 1\le k \le n\}$, and 
 $s\to\mathcal{Z}_{N,n}(s,w,\M)$ has poles of order at most $Z+1$.
Poles may occur only at rational  numbers   $\tilde{s}\le N/\big(n-Z\big)$,  $\tilde{s}=a/b$  for some $a,b\in\Z$ and $n-Z\le b\le n$.

Moreover, $\mathcal{Z}_{N,n}(s,w,\M)$ is analytic  at $(-\ell,w)$ for all non-negative integers $\ell$ and all $w\in\C^n$ with $\re(w_k)>0 \, \ (1\le k \le n)$.
\end{proposition}
\noindent  If we take $Z $ minimal, \ie $Z:= Z_\M$, we find of course the best  information on the order and location of the poles.
\begin{proof}
Since $M\ge N$ can be taken arbitrarily large in Lemma \ref{MeromorphyZ}, it suffices to prove the claims in   Proposition \ref{MeromZ} when $\re(s)>(N-M)/n$. By \eqref{InjDecomp} and \eqref{IDef},  
\begin{equation}\label{GammaI}
\mathcal{Z}_{N,n}(s,w,\M)=\Gamma(s)^{-n}\cdot\sum_{\gamma\in\InZ}I_{\M^\gamma,Z}(s,w^\gamma).
\end{equation} 
Thus, it suffices to prove that $\Gamma(s)^{-n}I(s,w)=\Gamma(s)^{-n}I_{\M,Z}(s,w )$ has   the properties of $\mathcal{Z}_{N,n}$ in Proposition \ref{MeromZ}. Using    \eqref{sjDef}      we can write the entire function in    \eqref{IMeromorphy}  as
\begin{align}\label{IMeromorphybis}
\frac{I(s,w)}{\Gamma(ns-N) } \Big(\prod_{p=0}^{M}\prod_{j=2}^{Z+1} \big(p+(n+1-j)s-{\textstyle{\sum_{k=j}^{Z+1}|F_{k}|}}\big) \Big)\prod_{p=0}^{M}(p+s)^{n-Z-1}.
\end{align}
Since $\Gamma(s)^{-1}$ is an entire function vanishing  only at non-positive integers, from   \eqref{IMeromorphybis}  it is clear that a singularity  $(\tilde{s},\tilde{w})$ of 
 $I(s ,w)$ can only occur  when $\tilde{s}=-p$ is a non-positive integer, 
or $p+(n+1-j)\tilde{s}-\sum_{k=j}^{Z+1}|F_k|=0$, or $n\tilde{s}-N$ is a non-positive integer.
Thus   $\tilde{s}$ has an expression $\tilde{s}=a/b$, $a,b\in\Z$, where $n-Z\le b\le n$, as claimed. 
 Suppose first that the pole $\tilde{s}=a/b$ is not a non-positive integer, so that the right-most 
 product in \eqref{IMeromorphybis}   does not vanish at $\tilde{s}$. 
 Thus $1/\Gamma(ns-N)$ or the double product in \eqref{IMeromorphybis} vanishes at $\tilde{s}$. 
But  for each of the $Z$ values of $j$ in   \eqref{IMeromorphybis}, 
at most one index $p$ can correspond to a factor vanishing at $\tilde{s}$, and only to order 1.
Since the factor $1/\Gamma(ns-N)$ likewise vanishes to order at 
most one, the poles of $s\to I(s,w)$ are of order at most $Z+1$, except possibly at a non-positive integers $\tilde{s}$ where the vanishing 
 could be to order $n$ due to the last product in \eqref{IMeromorphybis}.
But $\Gamma(s)^{-n} $ vanishes to order $n$ at non-positive 
 integers, so $\Gamma(s)^{-n}\cdot I(s,w)$ is regular there.
Proposition \ref{MeromZ}  now follows from \eqref{GammaI}.
\end{proof}
\begin{proof}[Proof of Lemma  
 $\mathrm{\ref{MeromorphyZ}}$.]
Using (\ref{yDef}-\ref{gDef})   it is clear that $I(s,w)$ is analytic in $(s,w)$ if $\re(w_j)>0$, $\re(ns-N)>0,$   and $\re(s_j)>0\, \  (2\le j\le Z+1)$. The inequalities on $s$ and $s_j$ hold if $\re(s)>N/(n-Z)$ as $\sum_{k=j}^{Z+1}|F_k|\le\sum_{k=1}^n|F_j|=N$ by \eqref{UnionFjM}. To get the meromorphic continuation of $I(s,w)$, we therefore assume always that $\re(w_j)>0\, \  (1\le j\le n)$, and for now that $\re(s)>N/(n-Z)$.

Since the integral expression \eqref{BasicInt} for $I$ does not in general converge for 
 $\re(s)\le N/(n-Z)$, we will integrate by parts to raise  the  exponents   of the $\sigma_j \, \  (1\le j\le n)$ in the integrand in \eqref{BasicInt}.
Integrating by parts over $\sigma_{n}$ in \eqref{BasicInt}, we get for $\re(s)>N$ (so 
 $\re(s_n)>0$ and  $g =g_{w',\M,Z} $ as in \eqref{gDef}),
\begin{align*}
&\int_{\sigma_{n}=0}^{1}\sigma_n^{s_n-1}  g(\sigma )\,d\sigma_{n}=\frac{g(\sigma_1, \ldots,\sigma_{n-1},1)}{s_n}-\frac{1}{s_n}\int_{\sigma_n=0}^1\sigma_n^{s_n} \frac{\partial g}{\partial\sigma_n}(\sigma )\,d\sigma_{n}\\
&=\frac{1}{s_n}\int_{\sigma_n=0}^1\! \!\sigma_n^{s_n} \big((s_n+1)
g(\sigma_1, \ldots,\sigma_{n-1},1)-\frac{\partial g}{\partial\sigma_{n}}(\sigma )\big)\,d\sigma_n =\frac{1}{s_n}\int_{\sigma_{n}=0}^1\! \!\sigma_n^{s_n}  g_{0}(s_n,\sigma  )\,d\sigma_{n},
\end{align*}
with the obvious definition of $g_{0}$.
Repeating the integration by parts $M$ more times, 
\[\int_{\sigma_{n}=0}^{1}\sigma_{n}^{s_n-1}  g(\sigma ) \,d\sigma_{n}=\Big(\prod_{p=0}^{M}\frac{1}{s_n+p}\Big) \int_{\sigma_{n}=0}^{1} \sigma_{n}^{s_n+M}  g_{M}(s_n,\sigma )\,d\sigma_n,\]
where $g_M$ is a finite sum of $\sigma_{n}$-derivatives of $g$ and some specializations of them at $\sigma_n=1$, with coefficients which are polynomials in $s$.
The same procedure applied to $\sigma_{n-1},\ldots, \sigma_{2}$ replaces each 
$\sigma_j^{s_j-1}\ \,(2\le j\le n)$  in \eqref{BasicInt} by $\sigma_j^{s_j+M}$.  
We conclude that
\begin{equation}\label{Icontinued}
I(s,w)=T_{M}(s)\int_{\sigma_1=0}^{\infty}\sigma_1^{ns-N-1}  e^{-\sigma_1w_1}\int_{\sigma'\in(0,1)^{n-1}}g_*(s,\sigma ) \prod_{j=2}^n\sigma_j^{s_j+M}\  d\sigma'\,d\sigma_1,
\end{equation}
where   
\begin{equation}\label{IntegratedByPartsI}
T_{M}(s):=  \prod_{p=0}^M\prod_{j=2}^n\frac{1}{s_j+p},
\qquad\sigma=(\sigma_1,\sigma'),\qquad g_*(s,\sigma)=\sum_{u}c_{u}(s)f_u(\sigma) ,
\end{equation}
  the $c_{u}(s)=c_{u,w,\M}(s)$ being polynomials in $s$ with coefficients depending on $w, \M$ and $Z$, and the $f_u$ being 
 higher partial derivatives of $g$ with respect to the $\sigma_j$, with
 possibly some of the $\sigma_j$ specialized to the value $1$. Lastly,  the $u$ range over some finite index set.

Next we raise the exponent of $\sigma_1$. The MacLaurin expansion of order $M$ in the single variable $\sigma_{1}$ of $f_u$, with the integral form of the remainder, gives 
\begin{equation}\label{TaylorExp}
f_u(\sigma_1,\sigma' )=\sum_{\ell=0}^M \frac{\sigma_{1}^\ell}{\ell!} \frac{\partial^\ell f_u}{\partial\sigma_{1}^\ell}(0,\sigma' )+\frac{\sigma_{1}^{M+1}}{M!}\int_{y=0}^{1}(1-y)^{M} \frac{\partial^{M+1}f_u}{\partial\sigma_{1}^{M+1}}(\sigma_{1}y,\sigma')\,dy.
\end{equation}
From \eqref{gDef} and \eqref{IntegratedByPartsI} we see that $|f_u(\sigma)|$ is bounded by a polynomial (depending on $u,s,w,\M$) in $\sigma_{1}$, uniformly for $(\sigma_{1},\sigma')\in[0,\infty)\times   [0,1]^{n-1}$.
Substituting   \eqref{TaylorExp} into \eqref{IntegratedByPartsI} and then into \eqref{Icontinued}, we find for $\re(s)>N/(n-Z)$,   
\begin{equation}\label{MeroCont}
\begin{aligned}
&I(s,w)=T_{M}(s)\sum_{u}c_{u}(s)\bigg(\sum_{\ell=0}^{M} 
\frac{\Gamma(ns-N+\ell)}{\ell! \! \; w_1^{ns-N+\ell}}\int_{\sigma'}   \frac{\partial^{\ell}f_u}{\partial\sigma_{1}^\ell}(0,\sigma')  
\prod_{j=2}^n\sigma_j^{s_j+M} \,d\sigma' \\
& \  + \int_{\sigma_{1}=0}^{\infty}\!\! e^{-\sigma_{1}w_{1}}\sigma_{1}^{ns-N+M}\!\!\!\int_{\sigma' }  \prod_{j=2}^n\sigma_j^{s_j+M}\!\!\int_{y=0}^{1}\!\!\!\frac{(1-y)^{M}}{M!} \frac{\partial^{M+1}f_u}{\partial\sigma_{1}^{M+1}}(\sigma_{1}y,\sigma')\,dy \,d\sigma\!\bigg).
\end{aligned}
\end{equation}
We now actually have our meromorphic continuation. Indeed, for  all  the integrals   in \eqref{MeroCont}  to be analytic in $s$,  it suffices to have
  $\re(s_j+M)>0\ \,( 1\le j\le n)$. If $Z+2\le j\le n$, this means 
$\re(s)>-M$, while for $1\le j\le Z+1$ by \eqref{IntegratedByPartsI} 
and \eqref{UnionFjM}, 
\[\re(s_j+M)=(n+1-j)\re(s)+M  - \sum_{k=j}^{n}|F_k|\ge(n+1-j)\re(s)+M-N.\]
Since $M\ge N$ in Lemma \ref{MeromorphyZ} by assumption, it  follows that all integrals in \eqref{MeroCont}  are analytic in the right half-plane $\re(s)>(N-M)/n$.
As the  terms preceding the integral on the first line of  \eqref{MeroCont}    become entire functions of $s$ on being multiplied  by $ (T_M(s)\Gamma(ns-N)\big)^{-1}$, we have proved Lemma  \ref{MeromorphyZ}. 
\end{proof}
On reviewing the proof  we see that the main point was to change variables from $T$ to $\sigma $ in \eqref{BasicInt} so that the singularity $\big($for small $\re(s)\big)$ of $H$ at $T=0$ takes a simpler form.
After that the only thing we need about $g(\sigma)$ in the new integral is its smoothness and that its partial derivatives are dominated by the exponential term $e^{-w_{1}\sigma_{1}}$ for $\sigma_{1}\in[0,\infty)$.   
 
\subsection{Analytic continuation of the  zeta series $\zeta_{N,n}$. }\label{Smercont}
We note that 
if we assume $\re(a_{ij})>0$ for all $i,j$ in  \eqref{ZetaSerDef}, as Shintani  did \cite{Shi},      then $s\to\zeta_{N,n}(s,w,\M)$
 has only simple poles \cite[\S3]{FR}.\footnote{\  However, even in the Shintani case, $\zeta_{N,n}$ will
 have  infinitely many poles if $n>1$. All  poles are rational numbers lying to the left 
 of the abscissa of convergence  \cite[Prop.\ 3.1]{FR}.} However, as Hypothesis $\Hh$ only assumes 
 $\re(a_{ij})\ge0$,    $ \zeta_{N,n} $ can have poles of higher order. The simplest example is 
 $\zeta(s)^n=\zeta_{n,n}(s,\mathbf{1},\text{I}_n)$, where  $\mathbf{1}\in\C^n$ has all entries   1,  and $\text{I}_n$ is the $n\times n$ identity matrix. Similarly, products of   Shintani-Barnes zeta functions are  of  the form  $\zeta_{N,n}$, so such   products can  have  quite a variety of poles   \cite[\S3]{FR}.

We now show that the proof of the analytic continuation of the zeta integral $\mathcal{Z}_{N,n}$ given in \S\ref{Zmercont}   applies  almost verbatim  to the zeta series $\zeta_{N,n}$.
The only difference will turn out to be that the function $g$ in \eqref{gDef} will be replaced by a slightly more complicated $\widetilde{g}$.  On letting $\N_0:=\N\cup\{0\}$ we have for $\re(s)>N/\big(n-Z\big)$, 
\begin{align}\nonumber 
& \Gamma(s)^{n}\cdot\zeta_{N,n}(s,w,\M)=\sum_{k\in \N_0^N}\int_{T\in(0,\infty)^n} \prod_{j=1}^n  T_j^{s -1}\cdot e^{-T_j( w_j+k_1a_{1j}+\cdots+k_N a_{Nj}) }\,dT 
\\ 
&=\int_{T\in(0,\infty)^n}  \Big(  \prod_{j=1}^n e^{-w_j T_j}\cdot T_j^{s -1} \Big) \Big( \sum_{k\in \N_0^N}   \prod_{i=1}^{N}   e^{-k_i\sum_{j=1}^n  a_{ij}T_j }  \Big)\,dT \nonumber 
 \\ &
= \int_{T\in(0,\infty)^n}  \Big(  \prod_{j=1}^n   e^{-w_j T_j}\cdot T_j^{s -1} \Big)  \Big( \prod_{i=1}^{N}  \sum_{k_i=0}^\infty    e^{-k_i\sum_{j=1}^n  a_{ij}T_j }  \Big)\,dT 
\nonumber 
\\
&
=\int_{T\in(0,\infty)^n}  \frac{\prod_{j=1}^n   
 e^{-w_j T_j}\,T_j^{s -1} }{ \prod_{i=1}^{N}\!\big( 1-  e^{-\sum_{j=1}^n   a_{ij}T_j } \big) }\,dT   \nonumber 
\\ &
= \int_{T\in(0,\infty)^n}  \frac{\prod_{j=1}^n   e^{-w_j T_j}\,T_j^{s -1}}{ \prod_{i=1}^{N}\!\big(\! \sum_{j=1}^n   a_{ij}T_j \big)}\cdot\Phi(T) \, dT=:  \int_{T\in(0,\infty)^n} \widetilde{H}(T,s,w,\M)  \, dT,\label{ZetawithGamma}
\end{align}
where $\widetilde{H}$  stands for the integrand to its left and 
\begin{equation}\label{Phidef}
\Phi(T) := \prod_{i=1}^{N} \varphi\Big(\sum_{j=1}^n   a_{ij}T_j  \Big)\ \ \ \big(  T\in(0,\infty)^n\big),\quad\ \ \ \ \varphi(z):=\frac{z}{1-e^{-z}}\  \ \ \big(\re(z)>0\big).
\end{equation}
Note that  by Hypothesis  $\Hh$  in \eqref{HypH}, 
  $\Phi:(0,\infty)^n\to\C$  extends  as a  smooth function to $(-\varpsilon,\infty)^n$ for some $\varepsilon>0$.
    Also,   partial derivatives $\partial^\alpha$ of any order satisfy  $|\partial^\alpha(\Phi)(T)|\le H_\alpha(\|T\|)$  for all $T\in (-\varpsilon,\infty)^n$, where   $H_\alpha(\|T\|)$ is some polynomial in the Euclidean norm of $T$.  Lastly,  we note that
 $\Phi(T)=\Phi_\M(T)$  depends on $\M=(a_{ij})$ but not on   $w$ or $s$.
\vskip.2cm
As in \eqref{GammaI} and \eqref{IDef}, we have from \eqref{ZetawithGamma}
\begin{equation}
 \begin{aligned} 
\zeta_{N,n}(s,w,\M)&=\Gamma(s)^{-n}  \sum_{\gamma\in\InZ}\widetilde{I}_{\M^\gamma,Z}(s,w^\gamma),\\ \widetilde{I}(s)&=\widetilde{I}_{\M,Z}(s,w):=  \int_{T\in \Delta }  \widetilde{H}(T,s,w,\M)\,dT.  \label{GammaTildeI}
\end{aligned}
\end{equation}
The change of variables from $T$ to $\sigma$ in \eqref{NewVariables}     applied to    \eqref{GammaTildeI} yields 
\begin{equation}\label{ZetaBasicInt}
\widetilde{I}(s)=\int_{\sigma_1=0}^{\infty}\sigma_1^{ns-N-1}\cdot e^{-\sigma_1 w_1}\!\int_{\sigma' }\!\widetilde{g}(\sigma_1,\sigma')\cdot\prod_{j=2}^n\sigma_{j}^{s_j-1 } d\sigma'\,d\sigma_1,
\end{equation}
\begin{equation}\label{tildegDef}
\widetilde{g}(\sigma ) 
:= g(\sigma)\,\Phi(\sigma_1,\sigma_1\sigma_2,\ldots,\sigma_1\cdots\sigma_{Z+1},\sigma_{Z+2}{\textstyle{\prod_{j=1}^{Z+1}\sigma_j}},\ldots,\sigma_{n}{\textstyle{\prod_{j=1}^{Z+1}\sigma_j}}) ,
\end{equation}
with $\Phi$  as  in $\eqref{Phidef}$ $\big(${\it{cf.\ }}\eqref{Tfromsigma} and \eqref{BasicInt}-\eqref{gDef}$\big)$.  If need be, we will write $\widetilde{g}_{w,\M,Z}$ for $\widetilde{g}$.

We obtain the analogue for $\zeta_{N,n}$ of Proposition  \ref{MeromZ} by simply replacing  $\mathcal{Z}_{N,n}$  by  $\zeta_{N,n}$.
\begin{proposition}\label{MeromZeta}  If $\M$ and  $Z$ are as in Lemma $\mathrm{\ref{MeromorphyZ}}$, then  
 $(s,w)\to\zeta_{N,n}(s,w,\M)$  in \eqref{ZetaSerDef} has a meromorphic continuation  to 
$\C \times\{w\in\C^n|\,\re(w_k)>0,\ 1\le k \le n\}$, and 
 $s\to\zeta_{N,n}(s,w,\M)$ has poles of order at most $Z+1$.
Poles may occur only at rational  numbers   $\tilde{s}\le N/\big(n-Z\big)$,  $\tilde{s}=a/b$  for some $a,b\in\Z$ and $n-Z\le b\le n$.

Moreover, $\zeta_{N,n}(s,w,\M)$ is analytic  at $(-\ell,w)$ for all non-negative integers $\ell$ and all $w\in\C^n$ with $\re(w_k)>0 \, \ (1\le k \le n)$.
\end{proposition}
\begin{proof} 
As remarked at the end of the previous subsection, the proof of Lemma  
\ref{MeromorphyZ}   depended on \eqref{BasicInt}, but only used   the smoothness of $g$ and the polynomial boundedness of its partial derivatives.  
As these properties are shared by $\widetilde{g}$ in \eqref{tildegDef}, 
we see from \eqref{ZetaBasicInt} that Lemma 
\ref{MeromorphyZ} still holds if we replace $I$ by $\widetilde{I}$ everywhere. Proposition \ref{MeromZeta} then  follows on  replacing in the proof of  Proposition  
\ref{MeromZ}   every occurrence of $\mathcal{Z}_{N,n}$ by $\zeta_{N,n}$, every $I$ by $\widetilde{I}$ and every $g$ by $\widetilde{g}$.
\end{proof}
\section{ Values of $\zeta_{N,n} $ and  $\mathcal{Z}_{N,n} $ at $s=0,-1,-2,\ldots$}\label{Specval} In   \eqref{GammaTildeI}  we have expressed $\zeta_{N,n}(s,w,\M)$   
as  $\Gamma(s)^{-n}$ times a finite  sum of  
$n$-dimen\-sion\-al Mellin transforms $\widetilde{I}(s,w)$ of   elementary expressions. As $\Gamma(s)^{-n}$ vanishes to order $n$ at non-positive integers 
 $s=-\ell$,   only   the polar part  of   $\widetilde{I}(s,w)$ blowing up at $s=-\ell$  to order  $n$ 
 contributes  to $\zeta_{N,n}(-\ell,w,\M)$. 
We will show in Theorem \ref{AppliedFiniteTaylor} below   that this leads to  a   formula for $\zeta_{N,n}(-\ell,w,\M)$ in terms of a finite Taylor expansion at the origin of an explicit elementary function. This is a 
 widely used  method in dimension 1 \cite[Lemma 4.3.6]{BH},  applied in higher dimensions  
by Cassou-Nogu$\grave{\text{e}}$s and then Colmez to deal with Shintani's zeta function \cite[Prop.\ 7]{CN} \cite[Lemma 3.3]{Col}.

We will need some notation. 
Define integers $\alpha_j$ and functionals  $D^{(q)}$ as
\begin{align} \label{alphajDef} 
&\alpha_j =\alpha_j(\ell,\M,Z) := 
\begin{cases}(n+1-j)\ell 
+ \sum_{k=j}^{Z+1}|F_k(\M)| &\  \text{if }     2\le j\le Z+1 , \\ 
\ell &\  \text{if }     Z+2 \le j\le n,
\end{cases}
\end{align}
\begin{align}
   & \label{DrDef}
D^{(q)}(h)=D_{\ell,\M,Z}^{(q)}(h):=\frac{1}{q!\,\alpha_2!\,\alpha_3!\,\cdots\,\alpha_n!}\cdot\frac{\partial^{q+ \sum_{j=2}^n\alpha_{j}} h}{\partial\sigma_1^q\,\partial\sigma_{2}^{\alpha_{2}}\,\,\partial\sigma_{3}^{\alpha_{3}}\cdots\,\partial\sigma_{n}^{\alpha_{n}} }\bigg|_{\sigma=0} , 
\end{align}
where   $h=h(\sigma)=h(\sigma_1,\ldots,\sigma_n)$ and the set  $F_k(\M)\subset\{1,\ldots,N \}$ is given by \eqref{DefFjM}.   
\begin{theorem}\label{AppliedFiniteTaylor}
Suppose $\M=(a_{ij})$ satisfies Hypothesis $\Hh$ in 
  \eqref{HypH}, $w=(w_1,\ldots,w_n)=(w_1,w')\in\C^n$ satisfies 
 $\re(w_j)>0\ \,(1\le j\le n)$,  $\ell$ is a non-negative integer,   and 
 $Z<n$ is a non-negative  integer such that 
no row of $\M$ has more than $Z $ vanishing entries.  
Then the  value $\zeta_{N,n}(-\ell,w,\M)$ of the analytic continuation 
 of the Dirichlet series defined in \eqref{ZetaSerDef} is 
\begin{equation}\label{ZetaSeratell}
 \zeta_{N,n}(-\ell,w,\M) 
 =\frac{(-1)^{N}(\ell!)^n }{\displaystyle{\prod_{j=0}^Z }(n-j)}  \sum_{\gamma\in \InZ}\! \!  \sum_{q=0}^{n\ell+N}\frac{ (-1)^q (w^\gamma_1)^{n\ell+N-q} }{(n\ell+N-q)!}D_{\ell,\M^\gamma,Z}^{(q)}(\widetilde{g}_{w^{\prime^\gamma},\M^\gamma,Z}),
\end{equation}
where 
$w^\gamma_1:=w_{\gamma(1)}, \ \big(  w^{\prime^\gamma}  \big)_j:=w_{\gamma(j)}\ \, (2\le j\le n)$,  $\M^\gamma:=(a_{i\gamma(j)})_{\substack{1\le i\le N\\ 1\le j\le n}}$, $D^{(q)}$ is given by \eqref{DrDef}, $\widetilde{g}_{w',\M,Z}$ 
 by \eqref{tildegDef}, and $\InZ$ is the finite set defined two lines after  \eqref{ZwithGamma}.

Similarly, letting $\mathcal{Z}_{N,n}(s,w,\M)$ be as in  \eqref{ZetaIntDef} and $g_{w,\M,Z}$ as in \eqref{gDef}, we have
\begin{equation}\label{ZetaIntatell}
\mathcal{Z}_{N,n}(-\ell,w,\M)
=\frac{(-1)^N (\ell!)^n }{\displaystyle{\prod_{j=0}^Z }(n-j)} \!    \sum_{\gamma\in \InZ}   \!  \! \sum_{q=0}^{n\ell+N}\! \frac{(-1)^q(w^\gamma_1)^{n\ell+N-q} }{(n\ell+N-q)!}D_{\ell,\M^\gamma,Z}^{(q)}(g_{w^{\prime^\gamma},\M^\gamma,Z}).
\end{equation}
\end{theorem}
\noindent A glance at \eqref{gDef}, \eqref{tildegDef},  \eqref{ZetaSeratell} and \eqref{ZetaIntatell}  shows that $\zeta_{N,n}(-\ell,w,\M)$ and 
$ \mathcal{Z}_{N,n}(-\ell,w,\M)$ lie in $\Q(\{a_{ij}\})[w]$, \ie they are   polynomial functions   of $w_1,\ldots,w_n$ having coefficients in the subfield $\Q(\{a_{ij}\})\subset\C$ generated by the coefficients  of $\M=(a_{ij})$.
 
Before proving Theorem \ref{AppliedFiniteTaylor}, we simplify 
 \eqref{ZetaIntatell} by  computing its  inner sum 
over $q$.
\begin{corollary}\label{CorToThm5} With notation and assumptions as in Theorem \ref{AppliedFiniteTaylor}, define  the functional 
\begin{equation}\label{NewD}
 \mathcal{D} (H) := \frac{(-1)^N (\ell!)^n }{ (n\ell+N)!  \,
 \big(\prod_{j=2}^n\alpha_j!\big) \big(\prod_{j=0}^Z  (n-j)\big)}
 \frac{\partial^{\alpha_2 +\cdots+\alpha_n}  H }{\partial\sigma_{2}^{\alpha_{2}}  \cdots \partial\sigma_n^{\alpha_n}  } 
\bigg|_{\sigma' =0}  .
\end{equation}
 Then  
\begin{equation}\label{SimpleVersionTh5}
\mathcal{Z}_{N,n}(-\ell,w,\M)
=    \sum_{\gamma\in \InZ}   \!  \!  \mathcal{D}\bigg( \frac{\big(w_1^{\!\gamma } +
h_{w^{\prime^\gamma}\!,Z}(\sigma')\big)^{n\ell+N}}{y_{\M^\gamma,Z}(\sigma')} \bigg),
\end{equation}
where  $\sigma':=(\sigma_2,\ldots,\sigma_n)$, $y_{\M^\gamma,Z}$ is defined in \eqref{yDef}, and 
\begin{equation}
 h_{w'\!,Z}(\sigma' ):=\sum_{j=2}^{Z+1} 
  w_j\prod_{k=2}^j\sigma_k  \  \ + \ \
\Big(\prod_{k=2}^{Z+1}\sigma_k \Big)\cdot \sum_{j=Z+2}^n
 w_j\sigma_j.
\end{equation}
\end{corollary}
\begin{proof}
 Theorem \ref{AppliedFiniteTaylor} shows that 
$\mathcal{Z}_{N,n}(-\ell,w,\M)$   equals
\begin{equation}\label{SimplifyingTh5}
   \sum_{\gamma\in \InZ}   \!  \! 
 \mathcal{D}\bigg( \sum_{q=0}^{n\ell+N}  \binom{n\ell+N}{q}(-1)^q(w_1^{\!\gamma })^{n\ell+N-q} 
 \frac{\partial^q} {\partial\sigma_1^q}
\Big(\frac{ e^{-\sigma_1 h_{w^{\prime^\gamma},Z}(\sigma' )}}{y_{\M^\gamma,Z}}  
 \Big)\Big|_{\sigma=(0,\sigma' )}\bigg),
\end{equation}
where $\binom{n\ell+N}{q}$ is a binomial coefficient and  $\sigma=(\sigma_1,\sigma')\in\R^n$. 
Now, 
\begin{align*}
\frac{\partial^q}{\partial\sigma_1^q}\Big(\frac{ e^{-\sigma_1 h(\sigma' )}}{y_{\M^\gamma,Z}}   \Big)\Big|_{\sigma=(0,\sigma' )} = \frac{(-1)^q (h(\sigma' ))^q\,
 e^{-\sigma_1  h(\sigma' )}}{y_{\M^\gamma,Z}}
 \Big|_{\sigma=(0,\sigma' )} =  \frac{(-1)^q \big(h (\sigma' )\big)^q}{y_{\M^\gamma,Z}},
\end{align*}
where $h=h_{w^{\prime^\gamma},Z}$. On substituting this in \eqref{SimplifyingTh5}, the binomial theorem  yields \eqref{SimpleVersionTh5}.  
\end{proof}
\begin{proof}[Proof of Theorem 
 $\mathrm{\ref{AppliedFiniteTaylor}}$.]
As the proofs for $\mathcal{Z}_{N,n} $ and $\zeta_{N,n} $ will  be similar, we give  first  the proof for the  simpler case of $\mathcal{Z}_{N,n}$, and then point out the  changes needed for  $\zeta_{N,n}$.   
Let 
\begin{align}\label{RwDef}
R_\ell (w)=R_{\ell,\M,Z}(w):=\frac{(-1)^N (\ell!)^n }{n(n-1)\cdots (n-Z)}\sum_{q=0}^{n\ell+N} \frac{(-1)^q w_1^{n\ell+N-q} }{(n\ell+N-q)!}D ^{(q)}(g ),
\end{align}
so that on the right-hand side of \eqref{ZetaIntatell} we find $\sum_\gamma R_{\ell,\M^\gamma,Z}(w^\gamma)$.
From  \eqref{GammaI}, 
$$
\mathcal{Z}_{N,n}(s,w,\M)=\frac1{\Gamma(s)^n}\sum_{\gamma\in\InZ} \  I_{ \M^\gamma,Z}(s,w^\gamma ) ,
$$
and from Proposition  
\ref{MeromZ}   we know that   $\mathcal{Z}_{N,n}$  is regular at $s=-\ell$.   Hence   to complete the proof of \eqref{ZetaIntatell} it suffices to 
show  
\begin{align}\label{IsFormula}
\lim_{s\to-\ell}\frac{I_{ \M,Z} (s,w)}{\Gamma(s)^n}=R_{\ell,\M,Z}(w).
\end{align}
 Letting   $\partial^A g  :=\frac{\partial^{|A|} g}{\partial \sigma_1^{A_1}\cdots\,\partial \sigma_n^{A_n}} $,  we can write the   multi-variable Taylor expansion about the origin (with remainder in integral form) of $g$ to order $k$ \cite[pp.\ 12--13]{Hor} as  
\begin{align}\label{ManyVariableTaylorExp}
&g( \sigma)=\sum_{\substack{{A}\in\N_0^n  \\ |{A}|\le k}}\frac{\sigma^{A}}{{A}!}  \partial^A g(0)+(k+1)\!\!\!\!\sum_{\substack{{A}\in\N_0^n  \\ |{A}|= k+1}}\frac{{\sigma}^A}{A!}
\int_{y=0}^{1}(1-y)^k \partial^A g(y\sigma)\,dy,
\\ \nonumber
{A}&:= (A_1,\ldots,A_n),\quad |A|:=\sum_{j=1}^n A_j, \quad
  \sigma^{A}:= \prod_{j=1}^n \sigma_j^{A_j}  ,\quad A! :=\prod_{j=1}^n (A_j!). 
\end{align} 
This finite Taylor expansion holds for any  smooth complex-valued  
function   on an open convex subset of $\R^n$ containing $0$ and $\sigma$. 

Substituting \eqref{ManyVariableTaylorExp}  into \eqref{BasicInt},  using $s_1 =ns-N$ from \eqref{sjDef}, we find  for $\re(s)\gg0$,
\begin{align}\nonumber 
I(s)&=\sum_{\substack{{A}\in\N_0^n  \\ |{A}|\le k}}
\frac{ \partial^A g ( 0 )}{A!}  \Big(\int_{\sigma_1=0}^\infty e^{-w_1\sigma_1}\sigma_1^{A_1+ns-N-1}\,d\sigma_1\Big)  \prod_{j=2}^n\int_{\sigma_j=0}^1\sigma_j^{s_j+A_j-1}\,d\sigma_j  
\\ \nonumber 
&\  \  + \! \sum_{\substack{{A}\in\N_0^n  \\ |{A}|= k+1}}\!\!\!\frac{k+1}{A!}\int_{\sigma_1=0}^{\infty}  e^{-w_1\sigma_1} \int_{\sigma'} \prod_{j=1}^{n}\sigma_{j}^{s_{j}+A_{j}-1}\int_{y=0}^{1}\!\!(1-y)^{k}\partial^A g(y\sigma)\,dy\,\,d\sigma'\!\,d\sigma_1\\
 \label{gTaylorExpanded}
&=\Big(\prod_{j=2}^{n} \frac{1}{s_j+A_j}\Big) 
\sum_{\substack{{A}\in\N_0^n  \\ |{A}|\le k}}\frac{ \partial^A g}{A!}(0) \frac{\Gamma(ns-N+A_1)}{w_1^{ns-N+A_1}}\ +\sum_{\substack{{A}\in\N_0^n  \\ |{A}|= k+1}}  \frac{k+1}{ A!} F_A( s),
\end{align}
where the (obvious) meaning of $F_A( s)$ is spelled out in \eqref{FADef} below.

To prove \eqref{IsFormula} we will need to compute some limits. Let
 $u :=n\ell+N-A_1$, so 
\begin{align}\nonumber
 \frac{\Gamma(ns-N+A_1)}{\Gamma(s)}=  \frac{\big[\Gamma(ns-N+A_1)(ns-N+A_1+u)\big]}{\big[(s+\ell)\Gamma(s)\big]} \, \Big[\frac{s+\ell}{ns-N+A_1+u}\Big].
\end{align}
Each of the three terms within brackets  above has a limit as $s\to-\ell$.  Indeed, an easy induction shows that for $m\in\N_0$ the residue of $\Gamma(s)$ at the (simple) pole $-m$ is $(-1)^m/m!\,$. Thus, 
$$
\lim_{s\to-\ell}(s+\ell)\Gamma(s)=\frac{(-1)^\ell}{\ell!},\qquad \qquad 
\lim_{s\to-\ell}\frac{s+\ell}{ns-N+A_1+u}=\frac1n.
$$
Letting $z:=ns-N+A_1 $ and recalling   $u :=n\ell+N-A_1 $,  yields 
\[\lim_{s\to-\ell}\Gamma(ns-N+A_1)(ns-N+A_1+u)=\lim_{z\to-u}\Gamma(z)(z+u)=\begin{cases}
0 &\text{if }A_1>n\ell+N,\\
\frac{(-1)^{u}}{u!}  &\text{if }A_1\le n\ell+N.
\end{cases}\]
Hence, 
\begin{equation}\label{GammansOverGamma}
\lim_{s\to-\ell}\frac{\Gamma(ns-N+A_1)}{\Gamma(s)}=
\begin{cases}
0\ &\text{if }A_1>n\ell+N,\\
\frac{\ell! \,(-1)^{(n+1)\ell+N-A_1}}{n(n\ell+N-A_1)!}\ &\text{if }A_1\le n\ell+N.
\end{cases}
\end{equation}
Next we compute another limit. From \eqref{alphajDef} and  \eqref{sjDef}  we obtain
\begin{align} \nonumber
\lim_{s\to-\ell} \frac{1}{\Gamma(s)(s_j+A_j)}&= \lim_{s\to-\ell} \frac{1}{\Gamma(s)(s+\ell)}\cdot\frac{( s+\ell)}{(s_j+A_j)}\\  
&=
\begin{cases}
\frac{\ell! \,(-1)^{ \ell }}{ n+1 -j}\ &\text{if }A_j =\alpha_j\text{ and } \ 2\le j\le Z+1, \\  \ell!(-1)^{\ell}\ &\text{if } A_j=\alpha_j\text{ and } Z+  2\le j\le n, \\
0\  &\text{otherwise}.
\end{cases} \label{OneOverGammasjAj}
\end{align}
 
We prove next that for $k:=n\big((n+1)\ell+N+1\big)$  we have $\big(${\it{cf.}}\ \eqref{gTaylorExpanded} and \eqref{RwDef}$\big)$
\begin{align}\nonumber
&\lim_{s\to-\ell}
\sum_{\substack{{A}\in\N_0^n  \\ |{A}|\le k}}\frac{ \partial^A g ( 0 )}{A!}  \, \frac{\Gamma(ns-N+A_1)}{w_1^{ns-N+A_1} \prod_{j=2}^n (s_j+A_j)}\cdot \Gamma(s)^{-n}  \\
& \quad=
\lim_{s\to-\ell}  \sum_{\substack{{A}\in\N_0^n  \\ |{A}|\le k}}\frac{ \partial^A g ( 0 )}{A!}  \Big(\prod_{j=2}^n \big(\Gamma(s)(s_j+A_j)\Big)^{\!-1}\cdot  \frac{\Gamma(ns-N+A_1) }{\Gamma(s)\ \, w_1^{ns-N+A_1}}  =R_\ell(w).\label{GammaTimesTaylor} 
\end{align} 
Indeed,  \eqref{GammansOverGamma} and  \eqref{OneOverGammasjAj}   imply that none of the $A=(A_1,\ldots,A_n)$ on the left-hand side of  \eqref{GammaTimesTaylor} 
 contribute to this limit unless $0\le A_1\le n\ell+N$ and $A_j =\alpha_j\ \,(2\le j\le n)$.
Each of these contributing indices $A$ appears in the expansion
as we have chosen $k$  large enough. Namely, 
\begin{align*} 
|(A_1,\alpha_2,\ldots,\alpha_n)|&\le n\ell+N+\sum_{j=2}^{Z+1}\Big((n+1-j)\ell+ \sum_{k=j}^{Z+1} |F_k(\M)|\Big)+(n-Z-1)\ell
\\    &\le n\ell+N+\ \sum_{j=2}^{Z+1}(n\ell+N)\ +\ n\ell\leq n(n\ell+N)+n\ell<k,
\end{align*}
where we used $Z<n$ and \eqref{UnionFjM}.
Using \eqref{GammansOverGamma} and \eqref{OneOverGammasjAj}   we find that  $A=(A_1, \alpha_2,\ldots,\alpha_n)$ appears in \eqref{GammaTimesTaylor}, contributing the term corresponding to $q=A_1$ in the sum defining  $R_\ell(w)$ in \eqref{RwDef}.

To complete the proof of  \eqref{IsFormula} we  will  show that the meromorphic continuation to  $\C$ of  each $F_A( s)$ with $|A|=k+1$, has a pole  at $s=-\ell$ of order at most $n-1$. Indeed, for $\re(s)\gg0$ by definition, 
\begin{align}\label{FADef}
&F_A(s):= \int_{\sigma_1=0}^{\infty}  e^{-w_1\sigma_1}\int_{\sigma'} \prod_{j=1}^{n}\sigma_{j}^{s_{j}+A_{j}-1} \int_{y=0}^{1}(1-y)^{k}\partial^A g(y\sigma)\,dy d\sigma' d\sigma_1\\
&=\int_{\sigma_1=0}^\infty\!\! e^{-w_1\sigma_1}
\int_{\sigma'}\!G_A(\sigma )\prod_{j=1}^n\sigma_{j}^{s_j+A_j-1}\,
  d\sigma' d\sigma_1 \nonumber \quad\Big(G_A(\sigma ):=\int_{y=0}^{1}(1-y)^{k}\partial^A g(y\sigma)\,dy\Big).
\end{align}
 Note that  $G_A(\sigma )=G_A(\sigma_1,\sigma' )$  
 is $C^\infty$ for $(\sigma_{1},\sigma')\in(-\varepsilon,\infty)\times(-\varepsilon,1+\varepsilon)^{n-1} $ for some $\varepsilon>0$, and is bounded above by a polynomial in $\sigma_1$, independently  of $\sigma' \in[0,1]^{n-1}$. 

We can now carry out the analytic continuation of $F_A( s)$ to   the right half-plane $\re(s)>-\ell-\frac{1}{n}$ by repeated integration by parts, just as in the proof of Lemma 
 \ref{MeromorphyZ}.
This time, however, we have the advantage that   $\re(s_j+A_j)>\frac{1}{n}>0$ for a least one $j$ in the range $1\le j\le n$, as we will now show.
Indeed, 

\begin{align*}
\sum_{j=1}^n\,\re (&s_j+A_j )>|A|  +   \sum_{j=1}^n\Big( (-\ell-\tfrac1n)(n+1-j)-\sum_{k=j}^{Z+1} |F_k(\M)|\Big)
\\ &\ge |A|  +   \sum_{j=1}^n\big( (-\ell-\tfrac1n)(n+1-j)-N\big)\ge 
 |A|   +\sum_{j=1}^n\big( (-\ell-\tfrac1n)n-N\big)
\\ 
&=|A|-n  (n\ell+1+N)=k+1-n  (n\ell+1+N)= 1+n\ell \ge1.
\end{align*}

If $\re(s_{j_0}+A_{j_0})>0 $ for some $j_0\ge2$, then to effect the meromorphic continuation of $F_A(s) $ in \eqref{FADef} to the half-plane $\re(s)>-\ell-\frac{1}{n}$ just as we did for $I(s)$ in \S2, we need not carry out any  integration by parts with respect to $\sigma_{j_0}$.
Thus,
$T_M(s)=\prod_{j=2}^n\,\prod_{p=0}^M\frac{1}{s_j+p}$
in \eqref{Icontinued} is replaced by
${\prod_{j=2}^{'n}} \prod_{p=0}^M\frac{1}{s_j+p},$ 
where the  product over $j$ omits $j=j_0$.
This implies that $F_A(s)$ has  poles of order at most $n-1$ at $s=-\ell$.
Thus $F_A(s)/\Gamma(s)^n$ vanishes as $s\to-\ell$ if $2\le j_0\le n$. 

If $j_0=1$, \ie if $\re(s_1+A_1)>\tfrac1n$, we   go through with  the integration by parts with respect to the $n-1$ variables $\sigma_2,\ldots,\sigma_n$, accruing a pole at $s=-\ell$ of order at most $n-1$. 
However, in this case the factor  $e^{-w_1\sigma_1}\sigma_1^{s_1+A_1-1}$ in 
 \eqref{FADef} is integrable over $(0,\infty)$ as $\re(s_1+A_1)>0$.  
This implies that the  integration over $\sigma_1$ contributes no additional pole at $s=-\ell$, showing again that  $F_A(s)/\Gamma(s)^n$ vanishes as $s\to-\ell$. This concludes the proof of Theorem  
\ref{AppliedFiniteTaylor} for $\mathcal{Z}_{N,n}$.

The   above proof   applies  verbatim to $\zeta_{N,n}$  on replacing   $g$ by  $\widetilde{g}$ and $I(s,w)$ by $\widetilde{I}(s,w)$,   
just  as   the proof of  Proposition  
\ref{MeromZeta} followed from that of Proposition  \ref{MeromZ}.
\end{proof}

\section{Relations between zeta series and integrals}\label{Relations} Despite the parallel proofs exhibited so 
 far, $\zeta_{N,n}$ and $\mathcal{Z}_{N,n}$ do differ in some respects. For example, the homogeneity property in $w$  of 
 $\mathcal{Z}_{N,n}$, namely
\begin{equation}\label{Homog}
\mathcal{Z}_{N,n}(s,\lambda w,\M) =\lambda^{N-ns}
\mathcal{Z}_{N,n}(s,  w,\M)\qquad\qquad\qquad\big( \lambda >0\big),
\end{equation}
does  not hold for  $\zeta_{N,n}$.
 To prove \eqref{Homog} for $\re(s)\gg0$, simply change  
 variables from $t$ in the integral \eqref{ZetaIntDef} defining $\mathcal{Z}_{N,n}$ to  $t':=\lambda^{-1} t $.  For     $s\in\C$  
outside the possible singularities   $\tilde{s}$  in Proposition 
 \ref{MeromZ},    \eqref{Homog} then follows by analytic continuation.  
 
On the other hand, the $N$  difference equations in $w$   satisfied by $\zeta_{N,n}$, namely 
\begin{equation}\label{DifEq}
\zeta_{N,n}(s,w+\M_i,\M)-\zeta_{N,n}(s,w,\M)=-\zeta_{N-1,n}(s,w,\M^{\widehat{i\,}})\qquad (1\le i\le N ), 
\end{equation}
fail for  $\mathcal{Z}_{N,n}$. In \eqref{DifEq},  $\M_i\in\C^n$ is the $i^{\text{th}}$-row of $\M$ and  $\M^{\widehat{i\,}}$ is the $(N-1)\times n$ matrix that results 
 after removing   $\M_i$ from $\M$.     When $N=1$, \eqref{DifEq} holds if we define 
\begin{equation}\label{NEq0case}
\zeta_{0,n}(s, w) :=\prod_{j=1}^{n}  (w_{j}^{-s}).
\end{equation}
  For $\re(s)\gg0$, \eqref{DifEq} is proved by   canceling 
  the terms with   $k_i\ge1$  in the sums 
 \eqref{ZetaSerDef} defining the left-hand side. Analytic continuation again implies  
\eqref{DifEq}  for all $s$ outside the polar set  in Proposition 
 \ref{MeromZeta}. 

The relation between zeta integrals  $\mathcal{Z}_{N,n}$ and Dirichlet  series 
$\zeta_{N,n}$ becomes much clearer when we restrict 
$w\in\C^n$ to a subspace, namely to the row-space of  $\M$ (regarding $w\in\C^n$ as a $1\times n$ matrix). 
To parametrize the row-space, define a linear function 
\begin{align}\label{DefW}
W=W_\M:\C^N\to\C^n,\qquad
 \big(W_\M(x_1,\ldots,x_N)\big)_j:=\sum_{i=1}^N x_i a_{ij}\quad(1\le j\le n).
\end{align}
Thus, $  W(x)=\sum_{i=1}^N x_i\M_i=x\M$ (multiplication of the $1
\times N$ matrix $x$ by the $N
\times n$ matrix $\M)$. 

To insure   $\re\big(W(x)_j\big)\ge0$, we  will assume that $x_i $ lies in the open angular sector  $|\arg(x_i)|<A_\M\ \,( 1\le i\le N )$,  where  
\begin{equation}\label{AMDef}
A_\M:=\frac\pi2-\max_{ a_{ij} \not=0} \{|\arg(a_{ij})| \}.
\end{equation}
Note that our standing hypothesis $\Hh$ in \eqref{HypH} 
on $\M=(a_{ij})$  yields $A_\M>0$.  

If we wish ensure the strict inequality $\re\big(W(x)_j\big)>0$, we need to also assume  that $\M$ has no column of zeroes.  This holds for $\M=\M_\g$ in \eqref{LieToShinBarn}  since   any positive root $\alpha\in\Phi^+$ in  \eqref{WittDef} is   of the form $\alpha=\sum_{i=1}^r d_i\alpha_i$, where the $\alpha_i$ are the simple roots, $d_i\in \N_0$ for all $i$  and some $d_{i_0}\in\N$. Thus  $(\lambda_{i_0} ,\alpha^{\!\lor})>0$.

Note that using  \eqref{NEq0case} 
  definitions \eqref{ZetaSerDef} and \eqref{ZetaIntDef} of $\zeta_{N,n}$ and $\mathcal{Z}_{N,n}$ can be re-written  
\begin{align}\label{AltDefZeta}
\zeta_{N,n}\big(s,w,\M \big)&=\sum_{k\in\N_0^N}\!\! \zeta_{0,n}\big(s, w+W_\M(k )\big) ,
\\ \label{AltDefZ}
\mathcal{Z}_{N,n}\big(s,w,\M \big) &=\int_{t\in(0,\infty)^N}\!\! \zeta_{0,n}\big(s,w+ W_\M(t )\big)dt.
\end{align}
We  now relate $\mathcal{Z}_{N,n}(s,w,\M)$ to $\zeta_{N,n}(s,w,\M)$.
\begin{proposition}\label{Raabe}  If   Hypothesis $\Hh$ in 
 \eqref{HypH}  holds for $\M$, if $\re(w_j)>0\ \,(1\le j\le n)$ 
and  if $s$ is not one of the possible poles 
 $\tilde{s}$ in Propositions  $\mathrm{\ref{MeromZ}}$  and   $\mathrm{\ref{MeromZeta}}$, then 
 \begin{align}\label{RaabeF}
 \int_{t\in[0,1]^N}\zeta_{N,n}\big(s,w+W(t),\M \big) dt =\mathcal{Z}_{N,n}\big(s,w,\M\big)\qquad(\mathrm{``Raabe \  formula"}). 
\end{align}
If, in addition,    no column of $\M$ vanishes  and  if $|\arg(x_i)|<A_\M  \   (1\le i \le N)$, then 
\begin{align}
 \frac{\partial^N \mathcal{Z}_{N,n}\big(s,W(x),\M\big)}{\partial x_1\partial x_2\cdots \partial x_N}  & = 
  \big( \Delta_{e_1}\circ\Delta_{e_2}\circ\cdots\circ\Delta_{e_N}) 
 \big(\zeta_{N,n}(s,W(x),\M)\big) \nonumber  \\ &  
=   (-1)^N 
 \prod_{j=1}^n \Big(\sum_{i=1}^N x_i a_{ij}\Big)^{-s}=  (-1)^N\zeta_{0,n}\big(s,W_\M(x)\big)   ,\label{EQQS}
\end{align}
where $\Delta_{e_i}$ is the difference operator  in (iii) of Theorem  $\mathrm{\ref{Liepolys}}$.
\end{proposition}
\begin{proof}   
 Quite generally by   Fubini's theorem, if $\int_{t\in(0,\infty)^N}|f (t)|\,dt<\infty$ then 
$$
\int_{t\in [0,1]^N}\Big(\sum_{k\in\N_0^N}  f( k+t)\Big)\,dt =
\sum_{k\in\N_0^N} \int_{t\in k+[0,1]^N} f( t) \,dt=\int_{t\in  (0,\infty)^N } f( t)\, dt.
$$
Applying this to  $f(t):=\zeta_{0,N}\big(s,w+W_\M(t)\big)$, \eqref{AltDefZeta} and \eqref{AltDefZ}
prove   \eqref{RaabeF} for $\re(s)\gg0$, and so  by analytic continuation for any   $s\not=\tilde{s}$.

We prove    \eqref{EQQS} next. Take $w:=W(x)$,  so that   $w+W(t)=W(x+t)$. Letting $g(x):=\zeta_{N,n}\big(s, W(x),\M \big)$, we have from   \eqref{RaabeF}
\begin{align*} 
\int_{t\in[0,1]^N}g(x+t)\,dt= \int_{t\in[0,1]^N}\zeta_{N,n}\big(s, W(x+t),\M \big) dt =\mathcal{Z}_{N,n}\big(s,W(x),\M\big)=:h(x). 
\end{align*}
 But if   two smooth functions $g$ 
 and $h$  are related by the Raabe operator, so  $h(x)=\int_{t\in[0,1]^N}g(x+t)\, dt$,
 then  we claim (see below) that 
\begin{align}\label{CLaim}\frac{\partial^N h}{\partial x_1\cdots\partial x_N}= 
 \Delta_{e_1}\circ\Delta_{e_2}\circ\cdots\circ\Delta_{e_N} (g),
\end{align}
proving  the first line in \eqref{EQQS}. 
 The second line in \eqref{EQQS} follows from  
$$
\big(\Delta_{e_1}\circ\Delta_{e_2}\circ\cdots\circ\Delta_{e_N}) 
 \big(\zeta_{N,n}(s,W(x),\M)\big)  = (-1)^N\zeta_{0,n}\big(s,W_\M(x)\big). 
$$ 
This   in turn  is proved by repeatedly using $W_\M(x+e_i)=W_\M(x)+ \M_i$ and    \eqref{DifEq}.

 The   claim \eqref{CLaim}  is proved by moving the differential operator  into the integrand in the Raabe operator,  observing that 
 $\frac{\partial }{\partial x_i} g(t+x)= \frac{\partial  }{\partial t_i}g(t+x)$,  and  carrying out the successive  iterated integrals.
\end{proof}

\section{Proof of claims concerning  $P_{\ell,\g}$ and $Q_{\ell,\g} $}  \label{ProofTheorem1}
 Theorem $1'$ below includes  Theorem \ref{Liepolys} in 
 the Introduction regarding $P_{\ell,\g}$, adds the 
 corresponding claims for $Q_{\ell,\g}$, and adds (vi) below 
 connecting   $P_{\ell,\g}$ to $Q_{\ell,\g} $.
 \begin{theorem*}\label{Liepolysprime} Let $\g$ be a 
 semisimple complex Lie algebra of rank $r$,  
let $n $  be the number of positive roots 
 in a root system  for $\g$,  let $\zeta_\g(s,x) $ be as in \eqref{GenHurwDef}, $\mathcal{Z}_\g(s,x)$ as in 
 \eqref{GenIntzetaDef},
 and   let $\ell=0,1,2,\ldots$.
 Then the series in \eqref{GenHurwDef} and the 
 integral in 
 \eqref{GenIntzetaDef} converge for $\re(s)>r$ and  $x=(x_1,\ldots,x_r)$ with $x_k>0 \, \ (1\le k\le r)$, and are analytic functions of $(s,x)$ there. They have meromorphic 
 continuations in $s$ to all  of $ \C$ which are  regular at $s=-\ell$. The special values $P_{\ell,\g}(x):=\zeta_\g(-\ell,x)$ and  
 $Q_{\ell,\g}(x):=\mathcal{Z}_\g(-\ell,x)$ are polynomials in  
$ x_1,\ldots,x_r $ with rational coefficients, have total degree $n\ell+r$ and  satisfy the following.
\vskip.2cm

\noindent$\mathrm{(0)}$    $P_{\ell,\mathfrak{sl}_2 }(x)=-B_{\ell+1}(x)/(\ell+1)$ and $Q_{\ell,\mathfrak{sl}_2 }(x)=-x^{\ell+1} /(\ell+1)$.
    \vskip.15cm 

\noindent$\mathrm{(i)}$   $P_{\ell,\g}(x)$ and $Q_{\ell,\g}(x)$ depend  only on the isomorphism class of $\g$, up to re-numbering  
 the  $x_i$. More precisely, if $\g'$ is isomorphic to $\g$,   there is a  permutation  $\rho$ of $\{1,\ldots,r\} $ making  $P_{\ell,\g'}(x)= P_{\ell,\g}(x^\rho)$, where  
  $ (x^\rho )_i:= x_{\rho(i)}\ \,(1\le i\le r)$. Similarly, $Q_{\ell,\g'}(x)= Q_{\ell,\g}(x^{\rho'})$ for some permutation $\rho'$.

    \vskip.15cm 

\noindent$\mathrm{(ii)}$  If $\g_1$ and $\g_2$ are semisimple algebras, then $P_{\ell,\g_1\times\g_2}(x,y)=P_{\ell,\g_1}(x)  P_{\ell,\g_2}(y)$ and $Q_{\ell,\g_1\times\g_2}(x,y)=Q_{\ell,\g_1}(x)  Q_{\ell,\g_2}(y)$, on conveniently numbering the variables.

    \vskip.15cm

\noindent$\mathrm{(iii)}$  Define commuting  difference operators  
 $(\Delta_{e_k} P)(x):=P(x+e_k)-P(x)   $, where $e_1,\ldots,e_r$  is the standard basis of $\R^r$. Then, with $\lambda_k$ and 
$\alpha^{\!\lor}$ as in  \eqref{WittDef}, 
\begin{align*}
\big( \Delta_{e_1}\circ\Delta_{e_2}\circ\cdots\circ\Delta_{e_r})(P_{\ell,\g})(x) =\frac{\partial^N Q_{\ell,\g}(x)}{\partial x_1 \cdots \partial x_N} 
 = (-1)^r \Big(\! \prod_{  \alpha\in\Phi^+}\! \sum_{k=1}^r  x_k( \lambda_k,\alpha^{\!\lor})\!\Big)^{\!\ell}\in\Z[x].
\end{align*} 
    \vskip.15cm 
\noindent$\mathrm{(iv)}$  $P_{\ell,\g}(\mathbf{1} - 
x)=(-1)^{n\ell+r}P_{\ell,\g}(x),$ where   $\mathbf{1}:=(1,1,\ldots,1)\in \R^r$.  
    \vskip.15cm 

\noindent$\mathrm{(v)}$    
$\displaystyle Q_{\ell,\g}(x)=\sum_{\substack{L=(L_1,\ldots,L_r)\in\N_0^r\\ L_1+\cdots+ L_r=n\ell+r}}
 a_L\prod_{i=1}^r x_i^{L_i}\,  $ and   $\,\displaystyle
P_{\ell,\g}(x)=\sum_{\substack{L=(L_1,\ldots,L_r)\in\N_0^r\\ L_1+\cdots+ L_r=n\ell+r}} a_L\prod_{i=1}^r B_{L_i}(x_i)$, where both expressions share the same coefficients $a_L=a_{L,\ell,\g}\in\Q$.
    \vskip.15cm 

\noindent$\mathrm{(vi)}$   $\displaystyle  Q_{\ell,\g}(x)=\int_{t\in[0,1]^r}P_{\ell,\g}(x+t)\,dt  $.
\end{theorem*}
\noindent Since the Bernoulli polynomials $B_m(t)$ satisfy $\int_0^1 B_m(t)\,dt=0$ for $m>0$, the Bernoulli polynomial expansion  in  (v)   implies  
  $\int_{t\in[0,1]^r}   P_{\ell,\g}(t) \,dt=0$. In fact, as $\deg(P_{\ell,\g})=n\ell+r$, the   Bernoulli
 expansion in  (v) is   equivalent to \cite[Lemma 5.1]{FR}  
$$
 \int_{t\in[0,1]^r}  \frac{\partial^{|J|}P_{\ell,\g}(t)}{\partial t_1^{J_1}\cdots\partial t_r^{J_r}}\,dt=0  \qquad\big(J=(J_1,\ldots,J_r)\in \N_0^r,\ 0\le |J|:=\sum_{i=1}^r J_i <n\ell+r\big) . 
$$ 
 \begin{proof}
In \eqref{LieToShinBarn} and \eqref{LieToInt} we saw that on letting 
$\big(\M_\g\big)_{i\alpha}:=
(\lambda_i,\alpha^{\!\lor})\in\N\cup\{0\}$, then    $\M_\g$ satisfies hypothesis 
 $\Hh$ and
$$
\zeta_\g(s,x) = \zeta_{r,n}(s,W(x),\M_\g),\qquad 
\mathcal{Z}_\g(s,x)=\mathcal{Z}_{r,n}(s,W(x),\M_\g). 
$$
Moreover, as remarked three lines after
\eqref{AMDef}, no column of $\M_\g$ vanishes.

The convergence and analyticity for $\re(s)>r$ and  
$x=(x_1,\ldots,x_r) \in(0,\infty)^r$  
of the series \eqref{GenHurwDef} defining $\zeta_\g(s,x) $, and of the integral  \eqref{GenIntzetaDef}  defining $\mathcal{Z}_\g(s,x)$,  
follow  from the final sentence of \S\ref{Converge}. Their 
meromorphic continuation and  regularity at $s=-\ell$ follow from Propositions 
 \ref{MeromZ} and  \ref{MeromZeta}. That $ Q_{\ell,\g}(x)$ and $  P_{\ell,\g}(x) $  are polynomials with coefficients in $\Q$
 follows from the remark immediately after the statement of 
Theorem  \ref{AppliedFiniteTaylor} combined with the fact that   
$x\to W(x)$  is a linear function  with  coefficients in $\Q$. 

By the homogeneity property 
\eqref{Homog}  applied at $s=-\ell$,  $Q_{\ell,\g}(\lambda x)=\lambda^{n\ell+r}  Q_{\ell,\g}(x)$ for $\lambda>0$. Since $Q_{\ell,\g}(x)$ is 
 not identically zero by \eqref{EQQS}, it follows that 
$ Q_{\ell,\g}(x)$ is a homogeneous polynomial of   degree $n\ell+r$. 

 The Raabe formula 
 \eqref{RaabeF}  at $s=-\ell$ with $w:=W(x)$ proves   (vi) for
 $x\in(0,\infty)^r$.  As both sides of (vi) are polynomials in $x$ with coefficients in $\Q$, (vi) therefore holds for all $x\in\R^r$. 

It is shown in \cite[Lemma 2.4]{FP} that the Raabe operator mapping a polynomial 
 $f(x)\to \int_{t\in[0,1]^r} f(x+t)\,dt$ is a  degree-preserving 
 $\R$-vector space  automorphism of $\R[x]$     taking  
 the basis $\big\{\prod_{i=1}^r B_{L_i}(x_i) \big\}_{L\in\N_0^r} $ of $\R[x]$ 
to the  basis $\big\{\prod_{i=1}^r x_i^{L_i}\big\}_{L\in\N_0^r} $. Thus (vi) implies  that $P_{\ell,\g}(x)$   has degree  $n\ell+r$ and  that  (v) holds. Claim (iv)  follows from  $B_m(1-x)=(-1)^mB_m(x)$ and (v). 

Since the entries $\big(\M_\g\big)_{i\alpha}:=
(\lambda_i,\alpha^{\!\lor})$  are non-negative integers, (iii) follows from  \eqref{EQQS}. 
From    $\mathcal{Z}_{\mathfrak{sl}_2} (s,x):= 
\int_0^\infty(x+t)^{-s}\,dt=  -\frac{x^{-s+1}}{1-s}$, initially valid for $\re(s)>1$ and $ x>0$, claim (0) for $Q_{\ell,\mathfrak{sl}_2}$  follows  by analytic continuation to $s=-\ell$. 
Claim (v) then gives claim (0)  for $P_{\ell,\mathfrak{sl}_2}$. Of course, (0) was  long been known.

We now turn to (i) and (ii), having proved all the other  statements in
 Theorem $1'$. Given a root system $\Phi\subset E$,  where $E $ is 
an $r$-dimensional Euclidean vector space spanned by $\Phi$, the definition \eqref{GenHurwDef} of $\zeta_\g(s,x)$   requires 
arbitrarily choosing a system of positive roots $\Phi^+\subset\Phi$. Associated to $\Phi^+$ there is a  unique base, \ie a subset of  $\Phi $ consisting of $r$ simple roots  \cite[\S10.2]{Hum}
which we label  (again, arbitrarily) 
 $\alpha_1,\ldots,\alpha_r$. This fixes  the 
 fundamental dominant weights $\lambda_1,\ldots,\lambda_r
\in E $ as the basis dual  to the basis of co-roots 
 $\alpha_1^{ \, \lor},\ldots,\alpha_r^{\, \lor}$ under the inner 
 product $ (\ \, ,\  )$ on $E$ \cite[p.\ 67]{Hum}. Notice that 
 the role of the coordinate $x_i$ of $x $ in \eqref{GenHurwDef} thus 
 depends  on an arbitrary ordering of the fundamental dominant weights, or equivalently  of the simple roots.

 We now show that a different  choice of $ {\Phi}^{\widetilde{{\phantom{.}}} }\subset\Phi$ of positive roots can  only permute   the variables $x_1,\ldots,x_r$. Suppose   that 
  we have numbered  $  \widetilde{\alpha}_1,\ldots,\widetilde{\alpha}_r$ the simple roots of  $ {\Phi}^{\widetilde{{\phantom{.}}} }$. 
As there is an element  $\tau$  of the Weyl group of $\Phi$     for which  $ {\Phi}^{\widetilde{{\phantom{.}}} } 
 =\tau\big(\Phi^{+}\big)$  \cite[p.\ 51]{Hum}, we have the equality of sets $\{\tau(\alpha_1),\ldots,\tau(\alpha_r)\}=\{\widetilde{\alpha}_1,\ldots,\widetilde{\alpha}_r\} $. Thus there is a permutation 
$\sigma\in S_r $ for which $ \widetilde{\alpha}_i= \tau(\alpha_{\sigma(i)})$. 
As elements of the Weyl group are   compositions of   reflections, $\tau $ is an isometry  and so   
$\widetilde{\alpha}_i^{\, \lor} =  \tau({\alpha_{\sigma(i)}}^{\!\lor}) $. 
As $ \big(\tau(\lambda_{\sigma(j)}),\widetilde{\alpha}_i^{\,\lor}\big)=
\big(\tau(\lambda_{\sigma(j)}),\tau({\alpha_{\sigma(i)}}^{\!\lor})\big)=  \big( \lambda_{\sigma(j)},{\alpha_{\sigma(i)}}^{\!\lor} \big)=\delta_{ij}$, the    fundamental dominant weights 
  for ${\Phi}^{\widetilde{{\phantom{.}}} }$ are  given by  
$ \widetilde{\lambda}_i  =\tau(\lambda_{\sigma(i)})$ 
 Letting $\rho:=\sigma^{-1}\in S_r$ and using 
  $ \sum_{m\in\N_0^r}f(m)=\sum_{m\in\N_0^r}f(m^\sigma)$, we have for $\re(s)>r$,
\begin{align*}
&\sum_{m\in\N_0^r}\prod_{\widetilde{\alpha}\in{\Phi}^{\widetilde{{\phantom{.}}} }} 
\Big({ \sum_{i=1}^r} (m_i+x_i) \widetilde{\lambda}_i , \widetilde{\alpha}^\lor\Big)^{-s}
 = \sum_{m\in\N_0^r}\prod_{ \alpha \in \Phi^+} 
\Big( \sum_{i=1}^r  (m^\sigma_i+x_i) \tau(\lambda_{\sigma(i)}) ,
 \tau(\alpha^\lor)\Big)^{-s}\\ &=\sum_{m\in\N_0^r}\prod_{ \alpha \in \Phi^+} 
\Big( \sum_{i=1}^r  (m_{\sigma(i)}+x_i)  \lambda_{\sigma(i)}  ,
  \alpha^\lor \Big)^{-s} =\sum_{m\in\N_0^r}\prod_{ \alpha \in \Phi^+} 
\Big( \sum_{i=1}^r  (m_i+x_{\rho(i)}) \lambda_i ,
 \alpha^\lor\Big)^{-s} .
\end{align*}
This shows  for $\re(s)>r$ that 
replacing $\Phi^+$ by  ${\Phi}^{\widetilde{{\phantom{.}}}} $ in   \eqref{GenHurwDef} amounts to replacing $x$ by $x^\rho$. By analytic continuation,     
 $\zeta_\g(s,x)$
does  not depend  (up to re-numbering the $x_i$) on the choice
of a system of positive roots $\Phi^+\subset\Phi$ nor on the ordering of the simple simple roots in $\Phi^+$. An analogous argument   for integrals works for $\mathcal{Z}_\g(s,x)$.

We can now prove (i), \ie that up to re-numbering   the 
$x_i$,   $ \zeta_\g(s,x)$  and $\mathcal{Z}_\g(s,x)$  depend 
 only  on the isomorphism class of the root system 
$\Phi\subset E$ attached to $\g$, and so depend  only on the isomorphism class 
of $\g$ \cite[pp. 75 and 84]{Hum}.
Suppose $\Gamma\subset F$ is a  root system isomorphic to $\Phi\subset E$. By definition \cite[p. 43]{Hum},    there is then  a  linear isomorphism  $f:E\to F$   (not in general an isometry) 
mapping $ \Phi$ onto $\Gamma$ and satisfying for all $\alpha,\beta\in\Phi$  the relation
\begin{equation}\label{RootSysIsoDef}
\frac{\big( \alpha , \beta \big)}{\big( \alpha ,  \alpha \big)}= 
 \frac{\big(f(\alpha),f( \beta)\big)}{\big(f(\alpha),f( \alpha)\big)}, 
\end{equation}
where we have  again used $ (\ \, ,\  )$  for the inner product on $F$. It is  routine to show, without even needing  \eqref{RootSysIsoDef}, that if 
 $\Phi^+\subset\Phi$ is a system of positive roots for $\Phi$, then $\Gamma^+:=f(\Phi^+)\subset \Gamma=f(\Phi)$ is a 
 system of positive roots for $\Gamma$. Since we have 
 already  shown that    the choice of a set of positive roots 
 within a given root system and a choice of the ordering of the simple roots  only affect  the numbering of 
 the variables $x_i$, to prove the isomorphism invariance 
 claimed in (i)  it suffices to show that there is no change when we  replace $\Phi^+$ by  $\Gamma^+$ in the definition of  $\zeta_\g(s,x)$ in  \eqref{GenHurwDef}    (and similarly for $\mathcal{Z}_\g(s,x)$ in  \eqref{GenIntzetaDef}).

One checks that if $\alpha_1,\ldots,\alpha_r$ are the 
 simple roots in $\Phi^+$, then  $f(\alpha_1),\ldots,f(\alpha_r)$ are the simple roots in $\Gamma^+$. 
 We   
 check next that if $\lambda_1,\ldots,\lambda_r$ are the fundamental dominant weights corresponding to 
 $\alpha_1,\ldots,\alpha_r$,  then  
 $f(\lambda_1),\ldots,f(\lambda_r)$ are the fundamental 
 dominant weights corresponding to  
$f(\alpha_1),\ldots,f(\alpha_r)$.   The $\lambda_i\in E,$ satisfy for $1\le i,j\le r$ the defining relation 
$ (\alpha_j^{\,\lor},\lambda_i)=\delta_{ij}$ (\,= Kronecker $\delta$).  Using the $\R$-basis 
$\alpha_1,\ldots,\alpha_r$ of $E$, we can write 
 $\lambda_i=\sum_k c_{ik}\alpha_k$, where  $c_{ik}\in\R$. Then,  
\begin{align*}
\delta_{ij}&= \big(\alpha_j^{\,\lor},\lambda_i \big)=  \big(\tfrac{2}{(\alpha_j, 
\alpha_j)}\alpha_j,{\textstyle{\sum_k}}  c_{ik}\alpha_k\big) =2\sum_k 
c_{ik}\frac{  (\alpha_j,  \alpha_k)}{(\alpha_j,\alpha_j)}
=2\sum_k 
c_{ik}\frac{  \big(f(\alpha_j),  f(\alpha_k)\big)}{\big(f(\alpha_j), f(\alpha_j)\big)}
\\ & 
= \big(\tfrac{2}{(f(\alpha_j),f(\alpha_j))}f(\alpha_j),\textstyle{\sum_k}  c_{ik}f(\alpha_k)\big) = 
\big( f(\alpha_j)^\lor, \textstyle{\sum_k} c_{ik}f(\alpha_k)\big) = \big( f(\alpha_j)^\lor,f(\lambda_i )\big),
\end{align*}
where we   used  \eqref{RootSysIsoDef} in the right-most equality of the first displayed line. Similarly,    
\begin{align*}
  \big({\textstyle{\sum_i}}&(m_i+x_i)f(\lambda_i), f(\alpha)^\lor\big) = \sum_{i,k} (m_i+x_i)c_{ik}\big(f(\alpha_k),\tfrac{2}{(f(\alpha), 
f(\alpha))}f(\alpha)\big)
 \\ 
&= \sum_{i,k} (m_i+x_i)c_{ik}\big( \alpha_k ,\tfrac{2}{(\alpha, 
\alpha )} \alpha \big)=\big(\textstyle{\sum_i}(m_i+x_i)\lambda_i, \alpha^\lor \big)\qquad\qquad(\forall\alpha\in\Phi^+),   
\end{align*}
showing that nothing changes when we replace  $\Phi^+$    by  $\Gamma^+$ in \eqref{GenHurwDef} or \eqref{GenIntzetaDef}, proving (i).

To prove (ii), suppose $\Phi_i\subset E_i$ is a root system 
for $\g_i \ \, (i=1,2)$. Then $\Phi=(\Phi_1,0)\cup (0,\Phi_2)\subset E:=E_1\times E_2$  is a root system for $\g_1\times\g_2$, where the  inner product on $E$ is the  sum of the 
 component-wise inner products.  As $\Phi^+=(\Phi_1^+,0)
\cup (0,\Phi_2^+)\subset\Phi$ is a system of positive roots, a glance at \eqref{GenHurwDef} and \eqref{GenIntzetaDef} now 
 shows that   (ii) holds. 
\end{proof}

\section{Examples}\label{EXAMPLES}
We conclude with   examples of $P_{\ell,\g}$ for small $\ell$ and $\g=\mathfrak{sl}_3, \mathfrak{sl}_4,\mathfrak{so}_5, G_2, \mathfrak{so}_7$ and 
$\mathfrak{sp}_6$.   
The polynomials below seem to have   no symmetries,  except   under     Dynkin diagram  automorphisms.   Simple $\g \not=\mathfrak{so}_{8}$  have    at most 2 such symmetries  \cite[p.\ 66]{Hum}. For   $\g= \mathfrak{sl}_{r+1}$ this gives   invariance under  $x_i\to x_{r+1-i}\ \,(1\le i\le r)$  in the examples below. 

 Note that by \eqref{RaabeFORM2}     any $P_{\ell,\g}$    
below becomes a
  $Q_{\ell,\g}$ on     replacing every $B_{L_i}(x_i)$ by $x_i^{L_i}$. We also note that our last two examples below correspond to dual root systems. Our calculations used PARI/GP to implement Theorem \ref{AppliedFiniteTaylor}.
\begin{align*}
 P_{0,\mathfrak{sl}_3}(x_1,x_2)=\frac{B_2(x_1)}{4} 
 +  B_1(x_1) B_1(x_2)
 + \frac{B_2(x_2)}{4}  \qquad  \qquad  \qquad  \qquad  \qquad  \qquad  \qquad  \qquad 
\end{align*}
\begin{align*}
& P_{1,\mathfrak{sl}_3}(x_1,x_2)=-\frac{B_5(x_1)}{60} 
 + \frac{B_3(x_1) B_2(x_2)}{6} 
 +  \frac{B_2(x_1) B_3(x_2)}{6} 
  -\frac{B_5(x_2)}{60}   \qquad  \qquad  \qquad  \qquad 
\end{align*}
\begin{align*} P_{2,\mathfrak{sl}_3}(x_1,x_2) =\frac{B_8(x_1) }{480} 
 + \frac{B_5(x_1)B_3(x_2)}{15}  
 + \frac{ B_4(x_1)B_4(x_2)}{8}
 + \frac{B_3(x_1)B_5(x_2)}{15}  
 + \frac{B_8(x_2)}{480}  \qquad 
\end{align*}
\begin{align*}
&
  P_{0,\mathfrak{sl}_4}(x_1,x_2,x_3)=  -\frac{ B_3(x_1)+B_3(x_3)}{30}   -
 \frac{B_2(x_1)B_1(x_2)+B_2(x_3)B_1(x_2)}{6}  - 
\frac{B_3(x_2)}{10} \qquad 
\\ &\qquad \qquad\qquad\quad \  \  \ - \frac{B_2(x_2) B_1(x_1) +B_2(x_2) B_1(x_3)}{3}
   - \frac{B_2(x_1)  B_1(x_3) +B_1(x_1) B_2(x_3)}{4} 
 \\ &\qquad \qquad \qquad\quad \  \ \ -B_1(x_1)B_1(x_2) B_1(x_3)   
\end{align*}
 \begin{align*}
&P_{0,\mathfrak{so}_5} (x_1,x_2)  =\frac{1}{2}B_2(x_1)
 + B_1(x_1) B_1(x_2)
 + \frac{1}{4}B_2(x_2)  \qquad  \qquad \qquad  \qquad \qquad  \qquad \qquad  \qquad
 \end{align*}
 \begin{align*}
&P_{1,\mathfrak{so}_5} (x_1,x_2) =  -\frac{1}{72}B_6(x_1)
 + \frac{1}{4}B_4(x_1) B_2(x_2)
 + \frac{1}{3}B_3(x_1) B_3(x_2)  \qquad  \qquad  \qquad  \qquad  \qquad  \qquad \qquad  \qquad
 \\
 & \qquad   \qquad   \qquad  \ \
 + \frac{1}{8}B_2(x_1) B_4(x_2)
 - \frac{1}{576}B_6(x_2)
 \end{align*}
\begin{align*}
 & 
P_{2,\mathfrak{so}_5} (x_1,x_2)  =  \frac{4}{525}B_{10}(x_1)
 + \frac{4}{21}B_7(x_1)B_3(x_2)
 + \frac{1}{2}B_6(x_1)B_4(x_2)+ \frac{13}{25}B_5(x_1)B_5(x_2) 
\\  &  \qquad   \qquad   \qquad 
 + \frac{1}{4}B_4(x_1)B_6(x_2) 
 + \frac{1}{21}B_3(x_1)B_7(x_2)  
 + \frac{1}{4200} B_{10}(x_2) 
&\\
\\ 
 &P_{3,\mathfrak{so}_5} (x_1,x_2)  =  -\frac{1}{1680}B_{14}(x_1)
 + \frac{1}{5}B_{10}(x_1) B_4(x_2) 
 + \frac{4}{5}B_9(x_1) B_5(x_2) \\
 &
 \qquad   \qquad   \qquad   \ \  + \frac{11}{8}B_8(x_1) B_6(x_2)  
 + \frac{9}{7}B_7(x_1) B_7(x_2) 
 + \frac{11}{16}B_6(x_1) B_8(x_2)  \\
 &
 \qquad   \qquad   \qquad   \ \   + \frac{1}{5}B_5(x_1) B_9(x_2) 
 + \frac{1}{40}B_4(x_1) B_{10}(x_2) 
 - \frac{1}{215040}B_{14}(x_2) 
\end{align*}
\begin{align*}
  P_{0,G_2} (x_1,x_2)   =  \frac{1}{4}B_2(x_1)  
 +B_1(x_1) B_1(x_2)  + \frac{3}{4}B_2(x_2)   \qquad  \qquad  \qquad  \qquad  \qquad  \qquad  \qquad 
\end{align*}
\begin{align*}
&P_{1,G_2} (x_1,x_2)  =  -\frac{151}{124416}B_8(x_1)  
 + \frac{1}{6}B_6(x_1) B_2(x_2)  
 +  B_5(x_1) B_3(x_2) 
 + \frac{5}{2}B_4(x_1) B_4(x_2) \\
 &
  \qquad   \qquad   \qquad   \ \,  + 3 B_3(x_1) B_5(x_2)  
 + \frac{3}{2}B_2(x_1) B_6(x_2)   
 - \frac{151}{1536}B_8(x_2) 
\end{align*}
\begin{align*}
&P_{2,G_2} (x_1,x_2)  = \frac{1}{12936}B_ {14}(x_1) 
 + \frac{4}{33}B_{11}(x_1) B_3(x_2) 
 + \frac{3}{2}B_{10}(x_1) B_4(x_2) \qquad  \qquad \qquad  \qquad \\
 &  \qquad   \qquad   \qquad   \   
 + \frac{77}{9}B_9(x_1) B_5(x_2) 
 + \frac{115}{4}B_8(x_1) B_6(x_2) 
 + \frac{3022}{49}B_7(x_1) B_7(x_2) \\
 &   \qquad   \qquad   \qquad   \ 
 + \frac{345}{4}B_6(x_1) B_8(x_2) 
 + 77B_5(x_1) B_9(x_2) 
 + \frac{81}{2}B_4(x_1) B_{10}(x_2) \\
 &  \qquad   \qquad   \qquad   \    
 + \frac{108}{11}B_3(x_1) B_{11}(x_2)
 + \frac{729}{4312}B_{14}(x_2) . 
\end{align*}
\begin{align*}
& P_{0,\mathfrak{so}_7}(x_1,x_2,x_3) =  -\frac{7}{96} B_3(x_1)- \frac{25}{96}  B_3(x_2) - \frac{ 1 }{24}  B_3(x_3) - \frac{ 1 }{3}   B_2(x_1) B_1(x_2) \qquad \qquad 
\\
& \qquad   \qquad \qquad  \quad \quad        \ \   - \frac{ 2 }{3} B_2(x_2) B_1(x_1) - \frac{ 1 }{4} B_2(x_1) B_1(x_3)  -\frac{ 1  }{2} B_2(x_2) B_1(x_3) 
 \\
& \qquad \qquad    \qquad  \quad \quad       \ \   -\frac{ 1 }{4} B_2(x_3) B_1(x_1) - \frac{ 1 }{4} B_2(x_3) B_1(x_2)  
 -  B_1(x_1)B_1(x_2)B_1(x_3)\\
\end{align*} 
\begin{align*}
&P_{0,\mathfrak{sp}_6}(x_1,x_2,x_3) =  - \frac{ 7 }{192} B_3(x_1)- \frac{ 25 }{192}  B_3(x_2) - \frac{ 1  }{6}  B_3(x_3) - \frac{ 1 }{6}   B_2(x_1) B_1(x_2) \qquad \qquad 
  \\
& \qquad \qquad    \qquad  \quad \quad       \ \ - \frac{ 1 }{3} B_2(x_2) B_1(x_1) - \frac{ 1 }{4} B_2(x_1) B_1(x_3)  -\frac{ 1 }{2} B_2(x_2) B_1(x_3)  \\
& \qquad \qquad    \qquad  \quad \quad       \ \  -\frac{ 1 }{2} B_2(x_3) B_1(x_1)  - \frac{ 1 }{2} B_2(x_3) B_1(x_2)   
 -  B_1(x_1)B_1(x_2)B_1(x_3) .
 \end{align*}

\end{document}